\documentclass[12pt,a4]{amsart}

\usepackage{amssymb,mathrsfs,amsmath,dsfont,array}% Using array, you can center the heading line of a tabular.
\usepackage{multirow,arydshln}
\usepackage{tikz}
\usetikzlibrary{arrows,positioning,decorations.pathmorphing,shapes}
\usepackage{capt-of}% It helps you to use caption for tikzpictures as well.
\usepackage{mathbbol}     % Using these you can use mathbb fonts for lowercase characters as well.
\usepackage{verbatim} %Allows use of \begin{comment} ... \end{comment}
\usepackage[all,cmtip]{xy}
\usepackage{amsbsy}% this helps us to use bold fonts for mathematical symbols as well. We then use $\pmb{any symbol}.$
\usepackage{amsmath}% this allows \tag{any symbol} to be used  for the equations.
\usepackage{nccmath}% this allows you to have equation with position on the left. \begin{fleqn}\begin{flalign}\end{fleqn}\end{fleqn}

\usepackage{hyperref}
\usepackage[normalem]{ulem}
\newtheorem{Thm}{Theorem}[section]
\newtheorem{lem}[Thm]{Lemma}
\newtheorem{Pro}[Thm]{Proposition}

\theoremstyle{definition}
\newtheorem{deft}[Thm]{Definition}

\theoremstyle{remark}
\newtheorem{rem}[Thm]{Remark}

\numberwithin{equation}{section}

\newcommand{\cA}{\mathcal{A}}
\newcommand{\cH}{\mathcal{H}}

\newcommand{\cN}{\mathcal{N}}
\newcommand{\cQ}{\mathcal{Q}}

\newcommand{\End}{\mathop{{\rm End}}\nolimits}

\newcommand{\id}{\mathop{\rm id}}

\renewcommand{\tilde}{\widetilde}

\def\a{\alpha}

\def\aa{\mathcal A}

\def\adot{\dot{\alpha}}

\def\andd{\quad\hbox{and}\quad}

\def\b{\beta}

\def\bs{\boldsymbol}

\def\bl4{B_{\ell\geq4}}

\def\bbbc{{\mathbb C}}

\def\d{\delta}
\def\D{\Delta}

\def\e{\epsilon}

\def\fg{\mathfrak{g}}

\def\scg{\mathscr{G}}

\def\hh{{\mathcal H}}

\def\fh{\mathfrak{h}}

\def\ii{\mathcal{I}}
\def\jj{\mathcal{J}}

\def\fk{\mathfrak{k}}

\def\lam{\lambda}

\def\LL{\mathcal{L}}

\def\fl{\mathfrak{L}}

\def\ep{\epsilon}
\def\fm{(\cdot,\cdot)}

\def\bbbr{{\mathbb R}}

\def\supp{\hbox{\rm supp}}

\def\1k{\frac{1}{k}}
\def\op{\oplus}
\def\ot{\otimes}

\def\sub{\subseteq}
\def\sg{\sigma}

\def\pf{\noindent{\bf Proof. }}

\def\sspan{\hbox{\rm span}}

\def\bbbz{{\mathbb Z}}

\def\1il{1\leq i\leq\ell}

\newcommand{\Bigop}[2]{\raisebox{0.1ex}{\scalebox{.7}{$\displaystyle \bigoplus_{#1}^{#2}\;$}}}

\begin{document}
\title{Zero-level integrable modules over twisted affine Lie superalgebras}
\thanks{$^\star$ Corresponding author.}
%\thanks{$^\ast$ The third author was supported by a grant from IPM ...}
\thanks{The third author's research was in part supported by a grant from IPM (No. 1403170414). }
\maketitle
\pagestyle{myheadings}

\markboth{}{}
\centerline{ Hajar Kiamehr$^{\rm a}$, Senapathi Eswara Rao$^{\rm b}$, Malihe Yousofzadeh$^{\rm a,c, \star}$}

\centerline{$^{\rm a}$ {\scalebox{0.65} {Department of Pure Mathematics, Faculty of Mathematics and Statistics, University of Isfahan,}}}\centerline{{\scalebox{0.65} { P.O.Box 81746-73441, Isfahan, Iran,}}}
\centerline{$^{\rm b}$ {\scalebox{0.65} {Tata Institute of Fundamental Research, Mumbai, India,}}}

\centerline{$^{\rm c}${\scalebox{0.65} { School of Mathematics
			Institute for Research in
			Fundamental Sciences (IPM), }}}
\centerline{{\scalebox{0.65} {P.O. Box: 19395-5746, Tehran, Iran.}
}}
\centerline{{\scalebox{0.65} {Email addresses:  hkiamehr@sci.ui.ac.ir (Hajar Kiamehr), senapati@math.tifr.res.in (Senapathi Eswara Rao),}}}	
\centerline{{\scalebox{0.65} {		 ma.yousofzadeh@sci.ui.ac.ir \& ma.yousofzadeh@ipm.ir (Malihe Yousofzadeh).			
}}}

\begin{abstract} 
	The main result of this paper is the characterization of zero-level integrable finite weight modules, over twisted affine Lie superalgebras. We prove that such a module is parabolically induced from a module which is obtained, in a prescribed way, from a module over a Lie algebra $\mathscr{L}$ which is either a  $\bbbz$-graded abelian Lie algebra or a direct sum of  a $\bbbz$-graded abelian Lie algebra and the so-called quadratic Lie superalgebra $\mathcal{Q}$. We give also a complete characterization of both finite dimensional simple  $\mathcal{Q}$-modules as well as bounded finite weight $\bbbz$-graded simple $\mathcal{Q}$-modules.
\end{abstract}

\section{introduction} 
It is more than 50  years when affine Lie (super)algebras have been defined and since then, there have been  a lot of attempts to characterize their simple modules; see  \cite{BL3}-\cite{CP2}, \cite{DG}, \cite{FS}-\cite{KW}, \cite{you8}-\cite{you10} among  others. 

An affine Lie superalgebra,  which is the super version of an affine Lie algebra, has a root space decomposition with respect to a finite dimensional Cartan subalgebra $\hh$. An affine Lie superalgebra with zero odd part is just an affine Lie algebra.  According to the classification, an affine Lie superalgebra with nonzero odd part is either of type $Y^{(1)}$ where $Y$ is the type of basic classical simple Lie superalgebra with nonzero odd part or of one of the types $D(m+1,n)^{(2)},A(2m,2n)^{(4)},A(2m-1,2n)^{(2)}$ or $A(2m-1,2n-1)^{(2)};$ see \cite{van-thes}. The 
corresponding  root system of an affine Lie superalgebra $\LL$ is decomposed into three parts: real roots (those roots  which are not self-orthogonal), imaginary roots (those roots which are orthogonal to all  roots) and non-singular roots (those roots which are self-orthogonal but not orthogonal to the entire root system); see \cite{DSY} and \cite{Y3} for the details.  
An $\LL$-module is called a weight module if it has a weight space decomposition $\LL=\Bigop{\a\in\hh^*}{}\LL^\a$ with respect to $\hh$ and called a finite weight module if all weight spaces are finite dimensional.
The  affine Lie superalgebra $\LL$ has a canonical central element $c\in \hh$ acting on a simple  weight module as a scalar called the level of a module.

There exist two  determinative issues in the characterization of simple finite weight modules; the first one is that if the level is zero or not and the second one is the behavior of the actions of  real root vectors.   One knows that real root vectors act on a simple finite weight module either injectively or locally nilpotently; this in turn gives different classes of $\LL$-modules. One of the important classes is the class of  finite weight modules on which real root vectors act locally nilpotently; such modules are called integrable modules. For almost all affine Lie superlagbras with nonzero odd part, there is no simple integrable finite weight modules of nonzero-level; see \cite{E1}, \cite{Y1}. The characterization of simple integrable finite weight modules of level zero over untwisted affine Lie superalgebras has been studied in \cite{Rao-Zhao} and \cite{Wu-Zhang}. In this work we characterize simple integrable finite weight modules of  level zero over twisted affine Lie superalgebras.

The set of imaginary roots of an affine Lie superalgebra $\LL$ is a free abelian group of rank one; say, e.g., $\bbbz\d.$  We show that the characterization of zero-level simple integrable finite weight modules are reduced to the characterization of zero-level simple $\hh$-weight modules over $\Bigop{k\in\bbbz}{}\LL^{k\d}.$ As the level is zero,  we actually deal with simple $\hh/\bbbc c$-weight modules over $\mathfrak{L}:=\Bigop{k\in\bbbz}{}\LL^{k\d}/\bbbc c.$ The algebra $\fl$ is a semi-direct product of a Lie superalgebra $\mathscr{L}$ with $\hh/\bbbc c$. To get the  characterization of simple integrable finite weight modules, we need to characterize $\bbbz$-graded simple $\mathscr{L}$-modules. In the non-super cases and all super cases other than type $A(2m,2n)^{(4)},$ $\mathscr{L}$ is a $\bbbz$-graded abelian Lie algebra whose $\bbbz$-graded simple $\mathscr{L}$-modules are known in the literature,  but in the type $A(2m,2n)^{(4)},$  $\mathscr{L}$ is a direct sum of a  $\bbbz$-graded abelian Lie algebra and a new $\bbbz$-graded Lie superalgebra which is not abelian; we call this new Lie superalgebra the quadratic Lie superalgebra. It is \[\cQ=\underbrace{s^2\bbbc[s^{\pm4}]\op t^2\bbbc[t^{\pm4}]}_{\cQ_0}\op \underbrace{t\bbbc[t^{\pm4}]\op t^{-1}\bbbc[t^{\pm4}]}_{\cQ_1}\]
with the bracket 
\begin{equation*}\label{bracket}\begin{array}{lll}
		~[\cQ_0,\cQ_0]=[t^2\bbbc[t^{\pm4}], \cQ_1]:=\{0\},& ~[t^{4k+1},t^{4k'-1}]:=0,&
		~[t^{4k+1},t^{4k'+1}]:=t^{4(k+k')+2},\\
		~[t^{4k-1},t^{4k'-1}]:=-t^{4(k+k')-2},&
		~[s^{4k-2},t^{4k'+1}]:=t^{4(k+k')-1},&
		~[s^{4k+2},t^{4k'-1}]:=t^{4(k+k')+1}.
	\end{array}
\end{equation*}
To obtain  our desired characterization, we need to characterize $\bbbz$-graded simple $\cQ$-modules with finite dimensional homogenous spaces.  We show that  these modules are kind of loop modules based on  finite dimensional simple $\cQ$-modules. Moreover, we show that a finite dimensional simple $\cQ$-module is in fact a finite dimensional simple module over a Clliford superalgebra with some compatibility property.

Here is the structure of the paper: Following this introduction, the paper has three additional sections. In Section~2, we prepare the necessary preliminaries for the main theorem, including characterizations of certain modules over graded commutative associative algebras and  Clifford superalgebras; we address these topics in dedicated subsections. In Section~3, we characterize $\bbbz$-graded simple $\cQ$-modules with finite dimensional homogenous spaces. The final section is devoted entirely to the proof of the main theorem.

\section{Preliminaries}
\subsection{Conventions} Throughout this paper, the underlying field for all vector spaces is the field $\bbbc$  of complex numbers and  all tensor products are taken over $\bbbc$ unless otherwise mentioned. The dual space of a vector space $V$ is denoted by $V^*.$ 
We make a convention that for a vector space $V$ and a symmetric or skew-symmetric bilinear form $\fm,$ we denote by `` $ \bar{~~} $ '', the canonical epimorphism $\bar{~~}:V\longrightarrow \bar V:=V/{\rm rad\fm}$; also, we denote the induced nondegenerate bilinear form on $\bar V$ by $\overline{\fm}.$ 

Assume  $L=L_0\op L_1,$ where  $\{0,1\}$ is the abelian group $\bbbz_2$, is an associative superalgebra (resp. a Lie superalgebra). A superspace $V=V_0\op V_1$ is called an  $L$-{\it module} if there is an algebra homomorphism $\pi$ from $L$ to  $\End(V)$ (resp. $\overline{\End(V)}=\End(V)$ that is endowed by the commutator bracket); we refer to $\pi$ as the {\it representation} corresponding to $V.$ 
For $L$-modules $V$ and $W$, by a module homomorphism  from $V$ to $W$, we mean a $\bbbz_2$-graded linear  transformation $\varphi:V\longrightarrow W$ with $\varphi(xv)=x\varphi(v)$ for all $x\in L$ and $v\in V.$

Next assume $G$ is an abelian group and the superalgebra $L$ is $G$-graded. An  $L$-module $V$ is said to be $G$-graded if $V$ is a $G$-graded  superspace and $L^\sigma V^{\sigma'}\sub V^{\sigma+\sigma'}$ for all $\sigma,\sigma'\in G$. $G$-graded submodules are defined in the usual manner. In this setting, $\supp_G(V)$ is defined to be the set of the elements $\sigma\in G$ for which $V^\sigma\neq \{0\}$. If there is no ambiguity, we denote $\supp_G(V)$ by $\supp(V)$. The $G$-graded $L$-module $V$ is called {\it $G$-graded simple} if trivial submodules are the only $G$-graded submodules of $V$ and it is called a {\it finite weight} module if all homogeneous spaces of $V$ are finite dimensional.  A finite weight $G$-graded $L$-module $V$ is called {\it bounded}, if the dimensions of  homogeneous spaces are bounded. 

\subsection{Graded associative algebras}\label{associative-MM} Assume $G$ is an abelian group and $\aa$ is a unital  commutative associative  $G$-graded  algebra. Suppose that
\begin{itemize}
	\item $K$ is a subgroup of $G$ and $V:=\sum_{\tau\in K}\bbbc x^\tau$ is a $\bbbz$-graded vector space with a basis $\{x^\tau\mid \tau\in K\}.$
	\item $\varphi\in \aa^*$ is an algebra homomorphism with $\varphi(\aa^\sg)=\{0\}$ if $\sg\not\in K$ and $\varphi(\aa^\sg)\neq \{0\}$ if $\sg\in K.$ 
\end{itemize}
Set $V(K,\varphi):=V$ and define 
\begin{align*}
	\cdot: &\aa\times V(K,\varphi)\longrightarrow V(K,\varphi)\\
	& (a,x^\tau)\mapsto \varphi(a)x^{\tau+\sigma}\quad(a\in \aa^\sigma,~\sigma\in G,~\tau\in K).
\end{align*}
Set $S:=\left <\supp(\cA)\right >$ to be the subgroup of $G$ generated by $\supp(\cA)$. Then, $V(K,\varphi)$ is a finite weight $S$-graded simple $\aa$-module.

Conversely, assume $\Omega$ is a finite weight $G$-graded simple $\aa$-module.  Fix a nonzero homogenous vector $v_0$ and define 
\begin{align*}
	\pi:&\aa\longrightarrow \Omega\\
	& a\mapsto av_0\quad(a\in \aa)
\end{align*}
Then, $I:={\rm ker}(\pi)$ is a maximal $G$-graded ideal of $\aa$ and as an $\aa$-module, $\Omega\simeq \aa/I.$ Since $V:=\cA/I$ is a $G$-graded simple unital commutative associative algebra and $\bbbc$ is algebraically closed, it  is isomorphic to  a twisted commutative group algebra, say as a vector space, $\cA/I=\Bigop{\tau\in K}{}\bbbc x^\tau$ where $K:=\left<\supp(\Omega)\right >.$ In particular, there is a linear functional $\varphi:\cA\longrightarrow \bbbc$ such that 
\[ax^\tau=\varphi(a)x^{\tau+\sigma}\quad(x\in \cA^\sigma,~ \sigma\in G,~ \tau\in K).\]
Since 
$a(bx^\tau)=(ab)x^\tau$ for $a,b\in \cA$ and $\tau\in k,$ we get that $\varphi$ is an algebra homomorphism. Moreover,  we have
\begin{align*}
	\varphi(\cA^\sigma)\Omega=\cA^\sigma\Omega=\cA^\sigma\sum_{\tau\in K}\Omega^\tau=\sum_{\tau\in K}\cA^\sigma\Omega^\tau\sub \sum_{\tau\in K}\Omega^{\sigma+\tau}=\{0\}\quad (\sigma\in G\setminus K).
\end{align*}
We then get, for $\tau\in K,$ that
\begin{align*}
	\sum_{\sigma\in K}\varphi(\cA^\sigma)\Omega^{\sigma+\tau}=\sum_{\sigma\in G}\varphi(\cA^\sigma)\Omega^{\sigma+\tau}=\cA\Omega^\tau=
	\Omega=\sum_{\sigma\in K}\Omega^{\sigma}=\sum_{\sigma\in K}\Omega^{\sigma+\tau}.
\end{align*}
This implies that $\varphi(\cA^\sigma)\neq \{0\}$ for $\sigma\in K.$ In particular, $\Omega\simeq V(K,\varphi).$
\subsection{Clifford superalgebras}\label{clifford}
Assume $r,s$ are two positive integers and set 

\noindent  $\bullet$ $Q(r)$ to be the associative superalgebra  consisting of all block matrices of the form 
\[\left(
\begin{array}{ll}
	\bs{a}&\bs{b}\\
	\bs{b}&\bs{a}
\end{array}
\right)
\]
where $\bs{a}$ and $\bs{b}$ are $r\times r$-matrices and 

\noindent $\bullet$ $M(r|s)$ to be the associative superalgebra  consisting of all block matrices 
\[\left(
\begin{array}{ll}
	\bs{a}&\bs{b}\\
	\bs{c}&\bs{d}
\end{array}
\right)
\]
where $\bs{a}$ and $\bs{d}$ are respectively, $r\times r$ and $s\times s$-matrices.
Up to isomorphism, $Q(r)$'s and $M(r|s)$'s are the only finite dimensional  simple associative unital superalgebras and up to isomorphism $\bbbc^{r\mid r}$ and $\bbbc^{r|s}$ are the only finite dimensional simple modules over $Q(r)$ and $M(r|s)$ respectively.  

\section{Finite weight $\bbbz$-graded simple $\cQ$-modules}\label{quadratic} Set $\cQ:=\underbrace{s^2\bbbc[s^{\pm4}]\op t^2\bbbc[t^{\pm4}]}_{\cQ_0}\op \underbrace{t\bbbc[t^{\pm4}]\op t^{-1}\bbbc[t^{\pm4}]}_{\cQ_1}$
and define \begin{equation*}\label{bracket}
	{\footnotesize \begin{array}{lll}
			~[\cQ_0,\cQ_0]=[t^2\bbbc[t^{\pm4}], \cQ_1]:=\{0\},& ~[t^{4k+1},t^{4k'-1}]:=0,&
			~[t^{4k+1},t^{4k'+1}]:=t^{4(k+k')+2},\\
			~[t^{4k-1},t^{4k'-1}]:=-t^{4(k+k')-2},&
			~[s^{4k-2},t^{4k'+1}]:=t^{4(k+k')-1},&
			~[s^{4k+2},t^{4k'-1}]:=t^{4(k+k')+1}.
	\end{array}}
\end{equation*}
We call $\cQ$, the {\it quadratic Lie superalgebra.}  Bounded finite weight $\bbbz$-graded simple  $\cQ$-modules  play an important role in the theory of integrable modules over twisted affine Lie superalgebras. So, 
one of the goals of this paper is to classify bounded finite weight $\bbbz$-graded simple  $\cQ$-modules.
Let us start with the following  two identities which are needed for our results and proved easily:  
For a  $\cQ$-module $V$, $r=\pm1,$ $\omega\in V,$ $k\in\bbbz^{>0},$ $p,l\in\bbbz$ and  $\tau=\pm1,$  we have 
\small{\begin{equation}\label{M--20}
		\begin{array}{l}
			\begin{array}{rl}
				(i)~t^{4 p -\tau}(s^{r(4l+2)}(s^{-r(4l+2)}\omega ))
				%&=[t^{4 p -\tau},s^{r(4l+2)}]s^{-r(4l+2)}\omega+s^{r(4l+2)}(t^{4 p -\tau}(s^{-r(4l+2)}\omega)) \\
				%&=s^{-r(4l+2)}([t^{4 p -\tau},s^{r(4l+2)}]\omega)+[[t^{4 p -\tau},s^{r(4l+2)}],s^{-r(4l+2)}]\omega\\
				%&+s^{r(4l+2)}([t^{4 p -\tau},s^{-r(4l+2)}]\omega)+s^{r(4l+2)}(s^{-r(4l+2)}(t^{4 p -\tau}\omega)) \\
				&=-s^{-r(4l+2)}t^{4(p+rl+r\d_{\tau,-r})+\tau}\omega+t^{4 p -\tau}\omega\\
				&-s^{r(4l+2)}t^{4(p-rl-r\d_{\tau,r})+\tau}\omega+s^{r(4l+2)}(s^{-r(4l+2)}(t^{4 p -\tau}w)),
			\end{array}
			\begin{comment}
				\begin{array}{rl}
					(i)~t^{4 p -\tau}(s^{4l+2}(s^{-4l-2}\omega ))
					&=[t^{4 p -\tau},s^{4l+2}]s^{-4l-2}\omega+s^{4l+2}(t^{4 p -\tau}(s^{-4l-2}\omega)) \\
					&=s^{-4l-2}([t^{4 p -\tau},s^{4l+2}]\omega)+[[t^{4 p -\tau},s^{4l+2}],s^{-4l-2}]\omega\\
					&+s^{4l+2}([t^{4 p -\tau},s^{-4l-2}]\omega)+s^{4l+2}(s^{-4l-2}(t^{4 p -\tau}\omega)) \\
					&=-s^{-4l-2}t^{4(p+l+\d_{\tau,-1})+\tau}\omega+t^{4 p -\tau}\omega\\
					&-s^{4l+2}t^{4(p-l-\d_{\tau,1})+\tau}\omega+s^{4l+2}(s^{-4l-2}(t^{4 p -\tau}w)),
				\end{array}
			\end{comment}
			\\\\
			{\begin{array}{rl}
					(ii)~(s^{r(4l+2)})^kt^{4p+\tau}\omega
					&=\displaystyle{\sum_{i=0}^k{k\choose i}t^{4(p+irl)+\tau+2ri}(s^{r(4l+2)})^{k-i}\omega}.
			\end{array}}
		\end{array}
\end{equation}}
It is not hard to prove the following lemma:
\begin{lem}\label{M--22}
	Suppose that $V$ is a $\bbbz$-graded   $\cQ$-module.
	\begin{itemize}
		\begin{comment}\item[(i)] Suppose $r=\pm1$ and  $W$ is a $\bbbz$-graded subspace of $V$ such that 
			\[\sum_{k\in\bbbz^{\geq 0}}\bbbc t^{r(2k+1)}W\sub W\andd \sum_{k\in \bbbz^{\geq 0}}\bbbc t^{r(4k+2)}W=\{0\}.\]
			If  $i> 1$ and  $v\in W$ is a nonzero homogeneous vector with $t^{4rj\pm1}v=0$ for all $1\leq j\leq i-1,$ then, there  is a nonzero homogeneous vector $v'\in W$ satisfying  $t^{4rj\pm1}v'=0$ for all $1\leq j\leq i.$ 
		\end{comment}
		\item[(i)] Suppose $r=\pm1$ and  $W$ is a nonzero $\bbbz$-graded subspace of $V$ such that 
		\[\sum_{k\in\bbbz^{\geq 0}}\bbbc t^{r(2k+1)}W\sub W\andd \sum_{k\in \bbbz}\bbbc t^{4k+2}W=\{0\}.\] Then, there is a nonzero homogenous vector $w$
		with $t^{\pm1}w=0.$ Moreover, if  $i_1,i_2\in\bbbz$ and  $v\in W$ is a nonzero homogeneous vector with $t^{4rj\pm1}v=0$ for all $i_1\leq j\leq i_2,$ then, there  are  nonzero homogeneous vectors $v',v''\in W$ satisfying  $t^{4rj\pm1}v'=0$ for all $i_1\leq j\leq i_2+1$ and    $t^{4rj\pm1}v''=0$ for all $i_1-1\leq j\leq i_2.$ 
		\item[(ii)] Assume  $u\in V$, $r,p=\pm1$ and  $l$  is a positive integer.
		\begin{itemize}
			\item[(1)] If $s^{-r(4l+2)}u=0$ and $ t^{4ri\pm 1}u=0$  for all positive integers $i$, then $ t^{4ri\pm 1}u=0$ for all integers $i$.
			\item[(2)] Suppose $N$ is a positive integer with  $N>2l+1$. If 
			$s^{r(4l+2)}u=0$ and $t^{4pi\pm 1}u=0$ for all $1\leq i\leq N-1$, then $t^{4ri\pm 1}u=0$ for all $i\in \bbbz^{>0}.$
		\end{itemize}
	\end{itemize}
\end{lem}
\begin{Pro}\label{M--28}
	Suppose $V$ is a nonzero bounded finite weight $\bbbz$-graded $\cQ$-module. If  $\displaystyle{\sum_{k\in\bbbz}}\bbbc t^{4k+2}V=\{0\}$, then, there is a nonzero homogeneous  vector $u\in V$ such that $t^{4k\pm 1}u=0$ for all  integers $k$.
\end{Pro}
\pf  Assume $d$ is a bound for the dimension of $\bbbz$-graded homogeneous  spaces. Using Lemma~\ref{M--22}(i), there is a nonero homogenous vector  $v\in V$, say e.g.,  of degree $p,$  such that  
$t^{4i\pm1}v=0$ for $-2d(2d+1)\leq  i\leq 2d(2d+1).$ 
If  $t^{4i\pm1}v=0$ for all $i\in\bbbz,$ we are done, otherwise there is $r=\pm1,$ $\sg=\pm1$ and a positive integer $N>2d(2d+1)$  such that 
\begin{equation}\label{w--MM}
	0\neq  w:=t^{4rN+\sg}v\in V^{p+4rN+\sg}  \andd  t^{4ri\pm1}v=0\quad (1\leq i\leq N-1).
\end{equation} 
Since $\dim(V^p)\leq d,$ there is 
$1\leq k\leq d$ such that $s^{r(4k+2)}s^{-r(4k+2)}v\in\sspan_\bbbc\{v, s^{r(4i+2)}s^{-r(4i+2)}v\mid 1\leq i\leq k-1\}.$ But  (\ref{M--20})(i) together with (\ref{w--MM})  implies that
\begin{align*}
	-s^{-r(4k+2)}w=-s^{-r(4k+2)}t^{4rN+\sg}v
	=t^{4r(N-k-\d_{\sg,-r})-\sg}(s^{r(4k+2)}(s^{-r(4k+2)}v))
\end{align*}
and \small\begin{align*}
	&t^{4r(N-k-\d_{\sg,-r})-\sg}v=0\quad\hbox{as well as}\quad t^{4r(N-k-\d_{\sg,-r})-\sg}s^{r(4i+2)}s^{-r(4i+2)}v=0 \quad(1\leq i\leq k-1).
	%\\
	%=&-s^{r(4i+2)}t^{4r(N-k+i+\d_{\sg,-1})+\sg}v+t^{4r(N-k)-\sg}v+s^{r(4i+2)}t^{4r(N-k-i-\d_{\sg,1})+\sg}v+s^{r(4i+2)}s^{-r(4i+2)}t^{4r(N-k)-\sg}v
	%=0.
\end{align*} So  
{\small \begin{equation*}\label{M--25}s^{-r(4k+2)}w=0\andd \left\{\begin{array}{l}
			t^{4ri+ \sg}w=t^{4ri+ \sg}t^{4rN+\sg}v=\sg t^{4r(N+i)+2\sg}v-t^{4rN+\sg}t^{4ri+ \sg}v=0,\\
			t^{4ri- \sg}w=t^{4ri-\sg}t^{4rN+\sg}v=t^{4rN+\sg}t^{4ri-\sg}v=0\quad(1\leq i\leq N-1).
		\end{array}\right.
\end{equation*}}
This together with  Lemma~\ref{M--22}(ii)(2) implies that  $t^{-4ri\pm 1}w=0$ for all $i\in\bbbz^{>0}.$
Now if $t^{4ri\pm 1}w=0$ for all $i\in\bbbz^{\geq 0}$  as well, we get the result; otherwise, we 
assume $N'\geq N$ is the  smallest positive  integer such that for some $\tau=\pm1,$
\[s^{-r(4k+2)}w=0,~~t^{4rN'+\tau}w\neq 0\andd t^{4ri\pm 1}w=0\quad(1\leq i\leq N'-1).\]  Since $d$ is a bound for the dimension of weight spaces, we have  $\dim(V^{p+4r(2k+1)d+4rN+4rN'+\sg+\tau+2r})\leq d.$ So, there is $1\leq \ell\leq d$ such that 
\[s^{r(4(2k+1)(d-\ell)+2)}t^{4r(N'+(2k+1)\ell)+\tau}w\in\sum_{j=0}^{\ell-1}\bbbc s^{r(4(2k+1)(d-j)+2)}t^{4r(N'+(2k+1)j)+\tau}w.\]
We make a convention that if $\ell-1<0$, the above summation is zero. This together with (\ref{M--20})(ii) implies that 
\begin{align*}
	s^{r(4(2k+1)(d-\ell)+2)}t^{4rN'+\tau}w=&s^{r(4(2k+1)(d-\ell)+2)}(s^{-r(4k+2)})^{2\ell}t^{4r(N'+(2k+1)\ell)+\tau}w\\
	=&(s^{-r(4k+2)})^{2\ell}s^{(r(4(2k+1)(d-\ell)+2)}t^{4r(N'+(2k+1)\ell)+\tau}w\\
	\in&\sum_{j=0}^{\ell-1}\bbbc (s^{-r(4k+2)})^{2\ell}s^{r(4(2k+1)(d-j)+2)}t^{4r(N'+(2k+1)j)+\tau}w\\
	=&\sum_{j=0}^{\ell-1}\bbbc s^{r(4(2k+1)(d-j)+2)}(s^{-r(4k+2)})^{2\ell}t^{4r(N'+(2k+1)j)+\tau}w=0.
\end{align*}
Set $$\ell':=(2k+1)(d-\ell)\andd u:=t^{4rN'+\tau}w\neq 0,$$ we have 
\[s^{r(4\ell'+2)}u=0,~~s^{-r(4k+2)}u=0\andd t^{4ri\pm 1}u=0\quad(1\leq i\leq N'-1).\]
By Lemma~\ref{M--22}(ii), we get that $t^{4ri\pm 1}u=0$ for all integers $i$, as we desired\footnote{We mention that the proof of this proposition has been motivated by \cite{CFR}}.
\qed

\subsection{$\bbbz$-graded simple modules over split extensions of $\cQ$}
Assume $A$ is a $\bbbz$-graded abelian Lie algebra and set 
\[\fg:=A\op \cQ\]
which is a Lie superalgebra with 
\[\fg_0=A\op \cQ_0\andd\fg_1=\cQ_1.\]
Our main goal is
to classify bounded  finite weight $\bbbz$-graded simple $\fg$-modules.
As we will see, to this end, we need first to classify finite dimensional simple $\fg$-modules.  Therefore, we first classify finite dimensional simple $\fg$-modules.
We have 
\begin{equation}\label{grading}\fg=\Bigop{k\in\bbbz}{}\fg^k \hbox{ with } 
	\fg^{k}=\left\{\begin{array}{ll}
		A^{k}\op\bbbc s^{k}\op \bbbc t^{k}& k\in 4\bbbz+2,\\
		A^{k}\op\bbbc t^{k} & k\in4\bbbz\pm1,\\
		A^{k} & k\in4\bbbz.
	\end{array}
	\right.
\end{equation}
Set 
\begin{equation}\label{fk}\fk:=A\op t^2\bbbc[t^{\pm4}]\op t\bbbc[t^{\pm4}]\op t^{-1}\bbbc[t^{\pm4}].
\end{equation}
Recall that $\fg_1=\fk_1$ and for   $\lam\in \fk_0^*,$ define
\begin{align}
	f_\lam:&\fk_1\times \fk_1\longrightarrow \bbbc\label{radflam}\\
	&(x,y)\mapsto \lam([x,y]) \quad(x,y\in\fk_1).\nonumber
\end{align}

We say $\lam$ is coradical-finite if $\fk_1/{\rm rad}(f_\lam)$ is finite dimensional.  We denote the  Clifford superalgebra  generated by odd elements $\bar x\in \overline{\fk_1}$ subject to the relations  $\bar x\bar y+\bar y\bar x=\bar f_\lam(\bar x,\bar y)$ for $x,y\in \fk_1$ by $\frak{C}_\lam$.

\medskip

\begin{lem}\label{degenerate}
	Suppose $\sg=\pm1$ and  $ \lam\in\fk_0^*$. If the restriction of $\lam$ to $t^2\bbbc[t^{\pm4}]$ is nonzero while $\lam(t^{r})=0$ for all $r\in (4\bbbz+2)\cap \sg\bbbz^{>0}$, then, $f_\lam$ is nondegenerate.
\end{lem}
\pf Without loss of generality, we assume $\sg=-1.$ Suppose  $k$ is  the smallest nonnegative integer  with  $\lam(t^{4k+2})\not =0$.
Assume $x$ is a  nonzero element of $\fk_1.$ We shall show $x$ is not in the radical of $f_\lam.$ We know that  $x$ is of the form $\sum_{i=1}^pa_it^{4m_i+j_i}$ in which  $a_1,\ldots,a_p$ are nonzero scalars,  $m_1\leq \cdots\leq m_p,$ $j_1,\ldots,j_p=1,-1$ and $m_i\neq m_{i'}$ if $i\neq i'$ and $j_i=j_{i'}.$
We have 
\begin{align*}
	f_\lam(t^{4(k-m_p+\d_{-1,j_p})+j_p},x )=&f_\lam(t^{4(k-m_p+\d_{-1,j_p})+j_p},\sum_{i=1}^pa_it^{4m_i+j_i} )=\lam([t^{4(k-m_p+\d_{-1,j_p})+j_p},\sum_{i=1}^pa_it^{4m_i+j_i} ])\\
	=&\sum_{i=1}^pa_i\lam([t^{4(k-m_p+\d_{-1,j_p})+j_p},t^{4m_i+j_i} ])
	=\sum_{i:j_i=j_p}a_ij_p\lam(t^{4(k-m_p+m_i+\d_{-1,j_p})+2j_p} )\\
	=&\sum_{i:j_i=j_p}a_ij_p\lam(t^{4(k-m_p+m_i)+2} )
	=a_pj_p\lam(t^{4k+2})+\sum_{p\neq i:j_i=j_p}a_ij_p\lam(t^{4(k-m_p+m_i)+2} )\\
	=&a_pj_p\lam(t^{4k+2})\neq 0,
\end{align*}
The last equality is due to the choice of $k$ and the fact that  $-m_p+m_i<0$ for all $1\leq i<p$.\qed

\begin{lem}[Finite dimensional simple $\fk$-modules] \label{lemfk} Suppose   $\lam\in\fk_0^*$ and recall (\ref{fk}) as well as  (\ref{radflam}). 
	\begin{itemize}
		\item[(a)] $\fk$-modules killed by ${\rm rad}(f_\lam)\cup \{x-\lam(x)\mid x\in \fk_0\}$ are in one to one correspondence with $\frak{C}_\lam$-modules.
		\item[(b)] Suppose that $U$ is a  nonzero $\fk$-module killed by $\{x-\lam(x)\mid x\in \fk_0\}$. 
		If $U$ is  finite dimensional, then, $\lam$ is coradical-finite and if $U$ is finite dimensional and simple, then, ${\rm rad}(f_\lam)U=\{0\}.$ 
		\item[(c)]  Up to isomorphism, there is unique  finite dimensional  simple $\fk$-module killed by  $\{x-\lam(x)\mid x\in \fk_0\};$  this unique module is identified with the unique  finite dimensional simple module over the Clifford superalgebra $\frak{C}_\lam$; we denote this unique module by $V_{_{c,\lam}}.$
	\end{itemize}
\end{lem}
\pf  
(a) The map $\varphi_\lam:\fk\longrightarrow \frak{C}_\lam$ with 
\begin{align*}
	x\mapsto\left\{
	\begin{array}{ll}
		\lam(x) & x\in \fk_0\\
		\bar x& x\in \fk_1
	\end{array}
	\right.
\end{align*}
satisfies the following for $x\in \fk_0$ and $y,z\in\fk_1:$
\begin{align*}
	\varphi_\lam[x,y]=0=\lam(x)\bar y-\bar y\lam(x)=\varphi_\lam(x)\varphi_\lam(y)-\varphi_\lam(y)\varphi_\lam(x)
\end{align*}
and 
\begin{align*}
	\varphi_\lam[y,z]=\lam([y,z])=f_\lam(\bar y,\bar z)=\bar y\bar z+\bar z\bar y=\varphi_\lam(y)\varphi_
	\lam(z)+\varphi_\lam(z)\varphi_\lam(y).
\end{align*}
So, we get an associative algebra epimorphism form $U(\fk)$ to $\frak{C}_\lam.$ This trivially induces an associative algebra epimorphism 
from  $U(\fk)/I_\lam$ to $\frak{C}_\lam$ 
%(in particular, $I_\lam$ is proper) 
in which $I_\lam$ is the ideal of $U(\fk)$ generated by $${\rm rad}(f_\lam)\cup \{x-\lam(x)\mid x\in \fk_0\}.$$
Also, the algebra homomorphism from the free unital associative algebra over $\fk_1/{\rm rad}(f_\lam)$ to $U(\fk)/I_\lam$ mapping $\bar x\in \bar \fk_1$ to $x+I_\lam$ induces an algebra homomorphism from $\frak{C}_\lam$ to $U(\fk)/I_\lam$.
This means that $\frak{C}_\lam$ is isomorphic to  $U(\fk)/I_\lam$ and so we are done.

(b) Assume $\Phi$ is the corresponding representation. Set  $K:={\rm im}(\Phi).$ We have $\Phi(\fk_i)=K_i$ for $i=0,1.$ If $x\in\fk_0$ and  $\Phi(x)=0,$ then, for each $v\in U$, we have  $\lam(x)v=x\cdot v=0$ and so $\lam(x)=0.$ Therefore, the functional $\lam$ induces a linear functional $\mu$ on $K_0$ with $\mu(\Phi(x)):=\lam(x).$   We also have 
\begin{align*}
	\mu([\Phi(x),\Phi(y)])=\mu(\Phi([x,y]))=\lam([x,y])=f_\lam(x,y)\quad(x,y\in \fk_1).
\end{align*}
Define  the  symmetric bilinear form $\fm_\mu$ on $K_1$  by \[(\Phi(x),\Phi(y))_\mu:=\mu([\Phi(x),\Phi(y)])\quad x,y\in \fk_1.\] We have 
\[f_\lam(x,y)=0\hbox{ if and only if } (\Phi(x),\Phi(y))_{\mu}=0.\] In particular, $x\in {\rm rad}(f_\lam)$ if and only if $\Phi(x)\in {\rm rad}\fm_{\mu};$ 
moreover, $\Phi$ induces a linear isomorphism from  $\fk_1/{\rm rad}(f_\lam)$  to $K_1/{\rm rad}\fm_{\mu}$ and so, if  $U$ is  finite dimensional, then,     $\lam$ is coradical-finite.

Next assume  $U$ is  finite dimensional  and simple. Let $x\in {\rm rad}(f_\lam)$ and  $u$ be an eigenvector -with corresponding eigenvalue $r$-  for the  $\fk$-module homomorphism 
\begin{align*}
	\Psi:&U\longrightarrow U\\
	&v\mapsto xv\quad(v\in U).
\end{align*}
We have    $xu=r u$ and so 
$$0=[x,x]u=2x(xu)=2r^2 u.$$
Therefore, $\{v\in U\mid xv=0\}$ is a nonzero   $\fk$-submodules of $U$ and so it is the entire $U$ as we desired.
%\[xyv=(-1)^{|y|}yxv+[x,y]v=0.\]
This completes the proof.

(c) follows from (a),(b) and general facts about module theory of Clifford superalgebras. \qed

\begin{Pro}[Finite dimensional simple $\fg$-modules]\label{fdfg}
	Assume $M$ is a finite dimensional simple $\fg$-module. Then, there is a coradical-finite linear functional $\lam\in\fk_0^*$ such that $xv=\lam(x)v$ for all $x\in \fk_0=A\op t^2\bbbc[t^{\pm4}]$ and $v\in M$.  Also, ${\rm rad}(f_\lam)M=\{0\}$ and $M$ 
	is a completely reducible $\fk$-module whose simple  constituents are isomorphic via even isomorphisms.
\end{Pro}
\pf
The first assertion easily follows  as $M$ is simple as well as finite dimensional and $A\op t^2\bbbc[t^{\pm4}]$ lies in the center of $\fg.$
Set $\mathscr{S}:=s^2\bbbc[s^{\pm4}].$ We have
\begin{equation}\label{inv-rad-1}
	[\mathscr{S},{\rm rad}(f_\lam)]=[s^2\bbbc[s^{\pm4}],{\rm rad}(f_\lam)]\sub {\rm rad}(f_\lam).
\end{equation}
Since $M$ is finite dimensional, we pick a simple $\fk$-submodule $V(1)$ of $M$. 
Using Lemma~\ref{lemfk}, we have 
${\rm rad}(f_\lam)V(1)=\{0\}$ and that $\lam$ is coradical-finite. If $\mathscr{S} V(1)=s^2\bbbc[s^{\pm4}]V(1)\sub V(1),$ then,  $M=V(1)$ and we are done, otherwise, there is $z_1\in \mathscr{S}$ such that $z_1 V(1)\not\sub V(1).$ Set $V(2):=V(1)+z_1V(1).$ For $x\in \fk,$ we have 
\begin{align*}
	xV(2)=x(V(1)+z_1V(1))\sub xV(1)+z_1 xV(1)+[x,z_1]V(1).
\end{align*}
This implies that $V(2)$ is a $\fk$-module; also  contemplating Lemma~\ref{lemfk}(b) together with  (\ref{inv-rad-1}), we have $xV(2)=\{0\}$ if $x\in {\rm rad}(f_\lam)$. In particular, using Lemma~\ref{lemfk}(a), we get that  $V(2)$ is a finite dimensional $\frak{C}_\lam$-module
and so there is a simple $\frak{C}_\lam$-submodule $W(2,1)$ of $V(2)$, which is in fact a $\fk$-submodule of $V(2)$, such that \[V(2)=V(1)\op W(2,1).\] Assume $\pi^{^2}_{_1}:V(1)\op W(2,1)\longrightarrow W(2,1)$ is the canonical projection map. Then, we have 
	\begin{align*}
		\pi^{^2}_{_1}(z_1(xv))=\pi^{^2}_{_1}(x(z_1v))+\pi^{^2}_{_1}([z_1,x]v)=\pi^{^2}_{_1}(x(z_1v))=x\pi^{^2}_{_1}(z_1v)\quad(x\in\fk,v\in V(1));
	\end{align*}
	that is, 
	\begin{align*}
		f:&V(1)\longrightarrow W(2,1)\\
		&v\mapsto \pi^{^2}_{_1}(z_1v)\quad(v\in V(1))
	\end{align*}
	is a nonzero $\fk$-module homomorphism of degree zero and so an isomorphism of degree zero. 
Now, if $M=V(2)$,  we are done, otherwise, assume  we have $z_1,\ldots,z_{r-1}\in\mathscr{S}$ for some positive integer $2<r$ such that for $V(i)=V(i-1)+z_{i-1} V(i-1)$ ($2\leq i\leq r-1$), we have  ${\rm rad}(f_\lam) V(i)=\{0\}$ for all $1\leq i\leq r-1$, and that  there are simple $\fk$-submodules $W(i,j)$ ($1\leq j\leq n_i$) of $V(i)$  $(2\leq i\leq r-1)$ which are isomorphic to $V(1)$ via isomorphisms of degree zero and 
\[V(i)=V(i-1)\op \Bigop{j=1}{n_{i}}W(i,j)\quad (2\leq i\leq r-1).\] If $V(r-1)=M,$ we are done, otherwise, there is $z_{r}\in\mathscr{S}$ with $z_r V(r-1)\not\sub V(r-1).$ Set 
\[V(r):=V(r-1)+z_r V(r-1).\] As above, we have ${\rm rad}(f_\lam) V(r)=\{0\}$ and so it is completely reducible, in particular, there are $n_r\in\bbbz^{>0}$ and simple $\fk$-submodules of $W(r,j)$ $(1\leq j\leq n_r)$ such that $V(r)=V(r-1)\op\Bigop{j=1}{n_r}W(r,j).$
Let $\pi_{_j}^{^r}:V(r)\longrightarrow W(r,j)$ ($1\leq j\leq n_r$) be the canonical projection map. For $1\leq j\leq n_r,$ there are $1\leq i_0\leq r-1,1\leq j_0\leq n_{i_0}$ such that ${{\pi_{_j}^{^r}}}(z_r W(i_0,j_0))\neq \{0\},$ so 
\begin{align*}
	f:&W(i_0,j_0)\longrightarrow W(r,j)\\
	&v\mapsto {{\pi_{_j}^{^r}}}(z_rv)\quad(v\in W(i_0,j_0))
\end{align*}
is a nonzero $\fk$-module homomorphism of degree zero and so an isomorphism of degree zero. 
This together with an induction process completes the proof.\qed
\medskip

\medskip
\begin{Pro}\label{s-c-w}
	Suppose that 
	\begin{itemize}
		\item {{$\lam\in \fk_0^*$ is a coradical-finite linear functional}} with corresponding finite dimensional  simple $\fk$-module $V:=V_{c,\lam},$ 
		\item $\varrho$ is a linear functional  on $\mathscr{S}=s^2\bbbc[s^{\pm4}]$,
		\item $\varphi:\mathscr{S}\longrightarrow {\rm End}(V)$ is a linear map satisfying \[\varphi(z)(xv)-x\varphi(z)v=[x,z]v\andd \varphi(z)\varphi(z')v-\varphi(z')\varphi(z)v=-\varrho(z,z')v\] for $x,y\in\fk,$ $z,z'\in \mathscr{S}$ and  $v\in V$.
	\end{itemize}
	Set \(\Omega=\Omega_{\lam,\varrho,\varphi}(V):=V\)  and define 
	\begin{align*}
		\cdot:&\fg\times \Omega\longrightarrow \Omega\\
		(x, v)&\mapsto\left\{
		\begin{array}{ll} xv & x\in \fk,\\
			\varrho(x)v+\varphi(x)v & x\in \mathscr{S}.
		\end{array}
		\right.
	\end{align*}
	Then, $\Omega$ is a finite dimensional simple $\fg$-module.
\end{Pro}
\pf It is trivial. \qed

\begin{Pro} 
	Any finite dimensional simple $\fg$-module is of the form $\Omega_{\lam,\varrho,\varphi}(V)$  as stated in Proposition~\ref{s-c-w}.
\end{Pro}
\pf Assume $M$ is a finite dimensional simple $\fg$-module. Using Proposition~\ref{fdfg}, there is a coradical-finite  linear  functional $\lam\in \fk_0^*$ such that $\fk_0$ acts as $\lam(\cdot)$ on $M$ and  $M=U\ot V$ in which $U$ is a vector space  and  $V$ is the unique finite dimensional  $\fk$-module annihilated by ${\rm rad}(f_\lam)\cup\{x-\lam(x)\mid x\in \fk_0\}$ and 
\[x(u\ot v)=u\ot xv\quad(x\in \fk, u\in U\andd v\in V).\]
Assume $\{u_1,\ldots,u_r\}$ is a basis for $U$. For each $z\in s^2\bbbc[s^{\pm4}]$ and $v\in V,$ there are $f^z_{i,j}(v)\in V$ such that 
\[z(u_i\ot v)=\sum_{j=1}^ru_j\ot f_{i,j}^z(v).\]
For $x\in \fk,$ we have 
\begin{align*}
	\sum_{j=1}^ru_j\ot xf_{i,j}^z(v)=xz(u_i\ot v)=zx(u_i\ot v)+[x,z](u_i\ot v)=&z(u_i\ot xv)+(u_i\ot [x,z]v)\\
	=&\sum_{j=1}^ru_j\ot f_{i,j}^z(xv)+(u_i\ot [x,z]v).
\end{align*}
This implies that 
\[xf_{i,i}^z(v)=f_{i,i}^z(xv)+[x,z]v\andd xf_{i,j}^z(v)=f_{i,j}^z(xv)\quad(j\neq i);\] in particular, for $i\neq j,$ $f_{i,j}^z$ is an even $\frak{C}_\lam$-module homomorphism and so it is a scalar, say e.g., $\lam_{i,j}^z.$ Also,  $f_{i,i}^z-f_{1,1}^z$ is an even $\frak{C}_\lam$-module homomorphism which implies that $f_{i,i}^z=\lam_{i,i}^z{\rm id}+f_{1,1}^z$ for some $\lam^z_{i,i}\in\bbbc.$
This implies that 
\[z(u_i\ot v)=((\sum_j\lam_{i,j}^zu_j)\ot v)+(u_i\ot f_{1,1}^z(v)).\]

Set $\varphi(z):=f_{1,1}^z:V\longrightarrow V$ and define
\begin{align*}
	\psi(z):&U\longrightarrow U\quad \quad u_i\mapsto \sum_{j=1}^r\lam_{i,j}^zu_j\quad(1\leq i\leq r).
\end{align*} 
For $x\in \fk,$ $z,z_1,z_2\in \mathscr{S}=s^2\bbbc[s^{\pm4}]$
%$u\in U$ 
and $v\in V,$ we have 
%\begin{align*}
%(\psi(z)u\ot xv)+(u\ot \varphi(z)(xv))=z(u\ot xv)=&z(x(u\ot v))\\
%=&x(z(u\ot v))+[z,x](u\ot v)\\
%=&x((\psi(z)u\ot v)+(u\ot \varphi(z)v))+(u\ot [z,x]v)\\
%=&(\psi(z)u\ot xv)+(u\ot x\varphi(z)(v))+(u\ot [z,x]v)
%\end{align*}
%This shows that 
\begin{equation*}\label{equality-6}
	\varphi(z)(xv)=x\varphi(z)(v)+[z,x]v.
\end{equation*}
We also have, for $z_1,z_2\in \mathscr{S}$, $x\in \fk$ and $v\in V,$ that 
%On the other hand using (\ref{equality-6}), we have 
%\begin{align*}
%\varphi(z_2)\varphi(z_1)(xv)=&\varphi(z_2)(x\varphi(z_1)(v)+[z_1,x](v))\\
%=& x\varphi(z_2)\varphi(z_1)(v)+[z_2,x]\varphi(z_1)(v)+[z_1,x]\varphi(z_2)(v)+[z_2,[z_1,x]]v
%\end{align*}
%So
\begin{align*}
	\varphi(z_2)\varphi(z_1)(xv)-\varphi(z_1)\varphi(z_2)(xv)
	%=& x\varphi(z_2)\varphi(z_1)(v)+[z_2,x]\varphi(z_1)(v)+[z_1,x]\varphi(z_2)(v)+[z_2,[z_1,x]]v\\
	%-&x\varphi(z_1)\varphi(z_2)(v)-[z_1,x]\varphi(z_2)(v)-[z_2,x]\varphi(z_1)(v)-[z_1,[z_2,x]]v\\
	%=& x(\varphi(z_2)\varphi(z_1)-\varphi(z_1)\varphi(z_2))(v)+[[z_2,z_1],x]v\\
	=& x(\varphi(z_2)\varphi(z_1)-\varphi(z_1)\varphi(z_2))(v)
\end{align*}
This means that $\varphi(z_1)\varphi(z_2)-\varphi(z_2)\varphi(z_1)$ is an even $\fk$-module homomorphism and so it is a scalar $\lam_{z_1,z_2}$. Therefore, this scalar  is zero as $V$ is finite dimensional and ${\rm tr}(\varphi(z_1)\varphi(z_2)-\varphi(z_2)\varphi(z_1))=0$.
On the other hand, for $z_1,z_2\in \mathscr{S}$, $x\in \fk$, $u\in U$ and $v\in V,$ we have  
%\begin{align*}
%z_1(z_2(u\ot v))=&z_1((\psi(z_2)u\ot v)+(u\ot \varphi(z_2)v))\\
%=&(\psi(z_1)\psi(z_2)u\ot v)+(\psi(z_2)u\ot \varphi(z_1)v)+ (\psi(z_1)u\ot \varphi(z_2)v)+(u\ot \varphi(z_1)\varphi(z_2)v).
%\end{align*}
%So, 
{\small \begin{align*}
		0=&z_1(z_2(u\ot v))-z_2(z_1(u\ot v))\label{heisenberg}\\
		%\\=&(\psi(z_1)\psi(z_2)u\ot v)+(\psi(z_2)u\ot \varphi(z_1)v)+ (\psi(z_1)u\ot \varphi(z_2)v)+(u\ot \varphi(z_1)\varphi(z_2)v)\nonumber\\
		%-&(\psi(z_2)\psi(z_1)u\ot v)-(\psi(z_1)u\ot \varphi(z_2)v)- (\psi(z_2)u\ot \varphi(z_1)v)-(u\ot \varphi(z_2)\varphi(z_1)v)\label{heisenberg}\\
		=&((\psi(z_1)\psi(z_2)-\psi(z_2)\psi(z_1))u\ot v)+(u\ot (\varphi(z_1)\varphi(z_2)-\varphi(z_2)\varphi(z_1))v).\nonumber\\
		=&(\psi(z_1)\psi(z_2)-\psi(z_2)\psi(z_1))u\ot v,
	\end{align*}
	and so we have 
	\[\psi(z_1)\psi(z_2)-\psi(z_2)\psi(z_1)=0.\]
	In particular, $U$ is a finite dimensional simple module over the abelian Lie algebra  $\mathscr{S}$ and so $U$ is one dimensional determined with a functional $\varrho.$
	\qed

	\subsection{Bounded finite weight $\bbbz$-graded simple  $\fg$-modules}
	
	\begin{lem}\label{sub-module-g}
		Recall (\ref{grading}) and suppose that $U$ is a  finite dimensional simple $\fg$-module  on which $t^2\bbbc[t^{\pm4}]$ acts nontrivially and set 
		\[V:=L(U):=U\ot\bbbc[t^{\pm1}].\]
		Define 
		{\begin{align*}
				\cdot:&\fg\times V\longrightarrow V\\
				& (x,u\ot t^n):=xu\ot t^{n+m} \quad\quad(x\in \fg^m).
			\end{align*}
		}
		The $\bbbz$-graded $\fg$-module $V$ has a  $\bbbz$-graded simple $\fg$-submodule.
	\end{lem}
	\pf  Since $U$ is a finite dimensional, it contains a simple $\fk$-submodule  $W$.  Also as it is simple, we have  $U=U(\fg)W=U(\mathscr{S})W$ where $\mathscr{S}=s^2\bbbc[s^{\pm4}]$. By   Proposition~\ref{fdfg}, there is a coradical-finite linear functional $\lam$ on $\fk_0$ such that $x\in \fk_0$ acts as $\lam(x)$  on $U$. For $x\in \fk_0,$ $z\in U(\mathscr{S})$ and $w\in W,$ we have 
	\(xzw=zxw=\lam(x)zw,\)
	which in turn implies that  $t^2\bbbc[t^{\pm4}]$ acts nontrivially on $W$.
	
	Since $\fk_1/{\rm rad}\fm_\lam$ is finite dimensional, $\fm_\lam$ is degenerate and so by Lemma~\ref{degenerate},  there are positive integers $r,s\in4\bbbz^{\geq 0}+2$ such that $\lam(t^r)$ and $\lam(t^{-s})$  are nonzero. Using the actions of $t^r$ and $t^{-s}$ on $U\ot t^k,$ we get injections 
	$t^{r}:U\ot t^k\longrightarrow U\ot t^{k+r}$ and  $t^{-s}:U\ot t^k\longrightarrow U\ot t^{k-s}$. Suppose  $M=\Bigop{k\in \bbbz}{}M^k=\Bigop{k\in \bbbz}{}(M\cap V^k)$ is a $\bbbz$-graded submodule of $V$, we have 
	\begin{align*}
		\dim(M^k)\leq \dim(M^{k+r})\leq \dim(M^{k+sr})\leq \dim(M^{k}).
	\end{align*}
	\begin{comment}
		The second inequality is due to applying  $t^r$ ($s$-times) and the third  inequality is due to applying $t^{-s}$ ($r$-times).
	\end{comment}
	This shows that $\dim(M^k)= \dim(M^{k+r})$ for all $k\in\bbbz$
	and so $\dim(M^{k+rj})=\dim(M^{k})$ for all $k\in\bbbz$ and  $j\in\bbbz^{\geq 0}.$ Set
	\[r_{_M}:=\sum_{i=0}^{r-1}\dim(M^i).\] If $N\subsetneq M$ are two $\bbbz$-graded submodules of $V$, there is $k\in \bbbz$ such that $\dim(N^k)<\dim(M^k)$ and so there is $0\leq k\leq r-1$ such that $\dim(N^k)<\dim(M^k)$. Therefore, we have $r_N< r_M.$
	\begin{comment}
		Note that each $k\in \bbbz$ is of the form $k=pr+s$ for some $0\leq s\leq r-1$. If $p>0,$ $\dim(M^s)=\dim(M^{s+pr})$ and if $p<0,$ $\dim(M^k)=\dim(M^{k-pr})=\dim(M^s).$
	\end{comment}
	
	Now, to the contrary assume $V$ does not contain a simple $\bbbz$-graded submodule. Therefore, we have a chain 
	$\cdots\sub N_3\subsetneq N_2\subsetneq N_1\subsetneq V$ of graded submodules of $V$ which in turn implies that 
	$\cdots<r_{_{N_3}}<r_{_{N_2}}<r_{_{N_1}}<r_{_V}.$ It is a contradiction and so we are done.\qed
	\begin{Pro}
		Assume $U$ is a finite dimensional simple $\fg$-module on which $t^2\bbbc[t^{\pm4}]$ acts nontrivially. Then, the  $\bbbz$-graded $\fg$-module $V:=L(U)$ is completely reducible in the sense that it is a direct sum of   $\bbbz$-graded simple  $\fg$-submodules.
	\end{Pro}
	\pf  Using Lemma~\ref{sub-module-g}, we pick a  simple $\bbbz$-graded submodule $M$ of $V$.
	Set
	\[U(k):=\{u\in U\mid u\ot t^k\in M^k\}\quad(k\in \bbbz).\]
	We have $\fg^r (U(k)\ot t^k)=(\fg^r U(k))\ot t^{k+r}$ for $r,k\in\bbbz.$ This shows that $\fg^r U(k)\sub U(k+r)$. So, $\sum_{k\in\bbbz}U(k)$ is a $\fg$-submodule of $U$ and so $U=\sum_{k\in\bbbz}U(k).$
	We have 
	\begin{align*}
		U\ot \bbbc[t^{\pm1}]=(\sum_{k\in \bbbz}U(k))\ot \bbbc[t^{\pm1}]=\sum_{{{r\in\bbbz}}}\sum_{k\in \bbbz}U(k)\ot t^{k+r}.
	\end{align*}
	Fix $r\in\bbbz.$ The map 
	\begin{align*}
		\varphi:&M\longrightarrow \sum_{k\in \bbbz}U(k)\ot t^{k+r}\\
		& u\ot t^k\mapsto u\ot t^{k+r}\quad (u\in U(k))
	\end{align*}
	is a $\fg$-module  isomorphism.\qed

	\medskip

	\begin{Thm}[{Characterization of bounded finite weight  $\bbbz$-graded simple $\fg$-module}]
		Suppose  that  {{$V=\op_{m\in \bbbz}V^m$}} is a  bounded finite weight  $\bbbz$-graded simple $\fg$-module. Then, one of the following happen:
		\begin{itemize}
			\item ${t^2\bbbc[t^{\pm4}]}$ acts non-trivially on ${V}$. In this case, there is a finite dimensional simple $\fg$-module $U$ such that $V$ is isomorphic to a simple component of $L(U).$
			\item ${t^2\bbbc[t^{\pm4}]}$ acts trivially on ${V}.$ In this case, $t^2{\bbbc[t^{\pm4}]}\op t{\bbbc[t^{\pm4}]}\op t^{-1}{\bbbc[t^{\pm4}]}$ acts trivially on ${V}$; in particular, $V$ is a finite weight $\bbbz$-graded simple module over the abelian $\bbbz$-graded Lie algebra $A\op s^2\bbbc[s^{\pm4}];$ see \S~\ref{associative-MM} 
		\end{itemize}
	\end{Thm}
	
	\begin{proof} We carry out the proof in the following two cases:
		\medskip

		\noindent $\bullet$ {\bf $\bs{t^2\bbbc[t^{\pm4}]}$ acts non-trivially on $\bs{V}:$} Suppose that $r\in 4\bbbz+2$ and $t^r:V\longrightarrow V$ is nontrivial. Since $[t^r,\fg]=\{0\}$ and $V$ is $\bbbz$-graded simple, $t^r$ is a $\fg$-module isomorphism, in particular, $V=t^r V.$ Set \[W:=\{t^rv-v\mid v\in V\}.\] Then,  $W$ is a $\fg$-submodule  of $V$; we emphasize that it is a $\fg$-submodule not a   $\bbbz$-graded $\fg$-module. We have $V=W+ \Bigop{i=0}{|r|-1}V^i.$ In particular, {{$V/W$}} is finite dimensional. Pick a $\fg$-submodule $U$ of $V$ containing $W$ such that $U/W$ is a maximal submodule of $V/W.$ We will get that $(V/W)/(U/W)\simeq V/U$ is a finite dimensional simple $\fg$-module. 
		
		\medskip

		Suppose that $\bar{~~}: V\longrightarrow V/U$ is the canonical projection map.  Then, $L(\bar V)=\bar V\ot \bbbc[t^{\pm1}]$ is a $\bbbz$-graded $\fg$-module with the action
		\begin{align*}
			\cdot:&\fg\times (\bar V\ot \bbbc[t^{\pm1}])\longrightarrow \bar V\ot \bbbc[t^{\pm1}]\\
			&(x,\bar v\ot t^k)\mapsto \overline{x v}\ot t^{s+k}\quad\quad(s,k\in\bbbz,~v\in V,~x\in\fg^s).
		\end{align*}
		We have 
		\[\bar V\ot \bbbc[t^{\pm1}]=\sum_{j\in\bbbz}\sum_{m\in \bbbz}(\overline{V^m}\ot t^{m+j})\] and that for $j\in\bbbz,$
		\begin{align*}
			\psi_j: &V\longrightarrow \sum_{m\in \bbbz}(\overline{V^m}\ot t^{m+j})\\
			&v\mapsto  \bar v\ot t^{m+j}\quad\quad\quad\quad(v\in V^{m})
		\end{align*}
		is a $\fg$-module  isomorphism of degree $j$. This means that, up to isomorphism, $V$ is a simple component of $\bar V\ot \bbbc[t^{\pm1}].$
		\medskip
		
		\noindent  $\bullet$ {\bf $\bs{t^2\bbbc[t^{\pm4}]}$ acts trivially on $\bs{V}:$}  In this case, $V$ is a nonzero bounded finite weight $\bbbz$-graded $\cQ$-module such that 
		$\Bigop{k\in\bbbz}{}\bbbc t^{4k+2} V=\{0\},$ so by  Proposition~\ref{M--28}, there is a nonzero homogeneous vector $u$ with $\Bigop{k\in\bbbz}{}\bbbc t^{2k+1} u=\{0\}.$ 
		So, $\{v\in V\mid \fg_1 v=\{0\}\}= \{v\in V\mid \fk_1 v=\{0\}\}$ is a nonzero graded  $\fg$-submodule. Therefore, $V$ is a finite weight $\bbbz$-graded simple module over the abelian Lie algebra $A\op s^2\bbbc[s^{\pm4}].$ 
	\end{proof}

	\subsection{Affine Lie superalgebras}\label{affine-MM}
	Suppose that $\mathscr{G}$ is a basic classical simple Lie superalgebra of type $Y$ with standard Cartan subalgebra ${H}$ and corresponding root system $\dot {\mathfrak{s}}.$ Assume $\fm$ is an invariant   nondegenrate supersymmetric even bilinear form on $\mathscr{G}.$ One knows that if  $Y=D(k+1,\ell) (\ell\neq 0), A(k,\ell) ((k,\ell)\neq (1,1),(0,0))$, then,  there is an automorphism $\dot \sg:\mathscr{G}\longrightarrow \mathscr{G}$ of order 
	\[l=\left\{\begin{array}{ll}
		4& Y=A(2m,2n)\\
		2&\hbox{otherwise}
	\end{array}
	\right.
	\] such that $\sg({H})\sub {H}.$ Throughout this subsection, we assume 
	\begin{align*}
		\sg:=\left\{
		\begin{array}{ll}
			\dot\sg &  Y=D(k+1,\ell), A(k,\ell),\\
			{\rm id} & \hbox{otherwise}.
		\end{array}
		\right.
	\end{align*}
	The automorphism $\sg$ induces a linear automorphism on the dual space ${H}^*$ of ${H},$  mapping $\a\in{H}^*$ to $\a\circ \sg^{-1}.$  By the abuse of notations, we denote this new automorphism by $\sg$ as well. 
	Considering the root space decomposition $\displaystyle{\mathscr{G}=\Bigop{\dot\a\in\dot{\frak{s}}}{}\mathscr{G}^{\dot\a}}$ of $\mathscr{G}$ with respect to ${H},$ we get the weight space decomposition 
	$\mathscr{G}=\Bigop{\a\in \dot R}{}\mathscr{G}^{(\a)}$  of $\mathscr{G}$ with respect to the fixed point subalgebra $\fh$ of $H$  under $\sg$, in which 
	\[\dot R=\{\pi(\dot\a):=\dot\a|_{_{\fh}}\mid \dot\a\in \dot{\frak{s}}\}\]
	and \[\mathscr{G}^{(\pi(\dot\a))}=\sum_{\substack{\dot\b\in\dot{\frak{s}}\\
			\pi(\dot\a)=\pi(\dot\b)}}\mathscr{G}^{\dot\b}=\sum_{i=0}^{l-1}\mathscr{G}^{\sg^i(\dot\a)}.\]
	Moreover, for the $l$-th primitive root $\zeta$ of unity and 
	\[{}^{[j]}\mathscr{G}:=\{x\in \mathscr{G}\mid \sg(x)=\zeta^jx\}\quad\quad(j\in\bbbz),\] we have 
	\[\mathscr{G}=\Bigop{j=0}{l-1}{}^{[j]}\mathscr{G}\andd 
{}^{[j]}\mathscr{G}=\Bigop{\a\in \dot{\mathfrak{s}}}{}{}^{[j]}\mathscr{G}^{(\a)}\quad{\rm with } \quad{}^{[j]}\mathscr{G}^{(\a)}={}^{[j]}\mathscr{G}\cap \mathscr{G}^{(\a)} \quad (j\in \bbbz).\]
	Assume $\bbbc c\op\bbbc d$ is a two dimensional vector space and set  \[\LL:=\Bigop{j=0}{l-1}({}^{[j]}\mathscr{G}\ot t^j\bbbc[t^{\pm l}])\op\bbbc c\op\bbbc d\andd \LL_c:=\Bigop{j=0}{l-1}({}^{[j]}\mathscr{G}\ot t^j\bbbc[t^{\pm l}])\op\bbbc c.\]
	The superspace  $\LL$ together with the bracket 
	$$[x\ot t^p+rc+sd,y\ot t^q+r'c+s'd]:=[x,y]\ot t^{p+q}+p\kappa(x,y)\d_{p+q,0}c+sqy\ot t^q-s'px\ot t^p$$
	is a Lie superalgebra called twisted affine Lie superalgebra of type $Y^{(l)}$ if $\sg=\dot\sg$ (remember that $l$ is the order of $\sg$) and it is called an untwisted affine Lie superalgebra of type $Y^{(1)}$ if $Y\neq A(\ell,\ell)$ and $\sg={\id}.$  The subspace $\LL_c$ of $\LL$ is an ideal of $\LL$ called the {\it core} of $\LL$ whose center is $Z(\LL_c)=\bbbc c.$ The quotient algebra $\LL_{cc}=\LL_c/\bbbc c$ is called the {\it centerless core} of $\LL.$
	\begin{rem}
		{\rm The differences between the cases $X\neq A(2m,2n)^{(4)}$ and $X=A(2m,2n)^{(4)}$ are significant in the discussions under consideration.
			More precisely, for all types other that type $X=A(2m,2n)^{(4)}$, we have $\pi(\dot\a)=\dot\a\mid_{\fh}=0$ ($\dot\a\in\dot{\frak{s}}$) if and only if $\dot\a=0.$ 
			This, in particular, gives that if $X\neq A(2m,2n)^{(4)},$ for $0\neq k\in \bbbz,$ we have $\LL^{k\d}={}^{[k]}\mathscr{G}^{(0)}\ot t^k\sub H\ot t^k$
			and so for $0\neq k,k'\in\bbbz,$ $[\LL^{k\d},\LL^{k'\d}]=\{0\}$ up to $\bbbc c,$ that is 
			\[\mathscr{L}:=\Bigop{0\neq k\in\bbbz}{}\LL^{k\d}\] is a $\bbbz$-graded abelian subalgebra of $\LL_{cc},$
			but it is not  the case for $X=A(2m,2n)^{(4)}.$ 
			In this regard, we need some information about type $X=A(2m,2n)^{(4)}$ which we will gather in \S~\ref{type 2m-2n}.}
	\end{rem}

	\medskip

	The subspace $\hh:=\fh\op\bbbc c \op\bbbc d$ is the standard Cartan subalgebra of $\LL;$  the root system of $\LL$ with respect to $\hh$ is 
	\[R=\{\pi(\dot\a)+k\d\mid k\in\bbbz, \dot\a\in\dot{\frak{s}}, {}^{[k]}\mathscr{G}^{(\pi(\dot\a))}\neq \{0\}\}.\]  
	In particular,
	\[
	\parbox{3in}{$R=\dot{\frak{s}}+\bbbz\d$ if $X=Y^{(1)}$ ($Y\neq  A(\ell,\ell)$).}
	\]
	For $\pi(\dot\a)+k\d\in R\setminus\{0\},$  we have \[\LL^{\pi(\dot\a)+k\d}={}^{[k]}\mathscr{G}^{(\pi(\dot\a))}\ot t^k\andd \LL^{0}=\fh\op\bbbc c\op\bbbc d.\]  We mention that the definition of an affine Lie superalgebra of type $A(\ell,\ell)^{(1)}$ is slightly different from the one stated for  the other types; for the details see \cite[Exa. 3.1]{DSY}. 
	The bilinear form on $\scg$ induces the following even  supersymmetric invariant nondegenerate bilinear form  
	\[(x\ot t^p+rc+sd,y\ot t^q+r'c+s'd)=\d_{p+q,0}(x,y)+rs'+sr'\] on $\LL.$ This bilinear form is nondegenerate on $\hh$; in particular, it naturally induces a symmetric nondegenerate bilinear form on $\hh^*$ denoted again by $\fm.$  
	If $\LL$ is a twisted affine Lie superalgebra, the root system $R$ of $\LL$ can be expressed as shown  in the following table:
	
	\medskip
	
	\begin{table}[h]\caption{Root systems of twisted affine Lie superalgebras} \label{table1}
		% "h" means that the table will be placed "here", if we use [t], it will be placed on "top" and [b] indicates the "below". If we do not use this command, the table is floating.
		{\footnotesize \begin{tabular}{|c|l|}
				\hline
				$\hbox{Type}$ &\hspace{3.25cm}$R$ \\
				\hline
				$\stackrel{A(2m,2n-1)^{(2)}}{{\hbox{\tiny$(m,n\in\bbbz^{\geq0},n\neq 0)$}}}$&$\begin{array}{rcl}
					\bbbz\d
					&\cup& \bbbz\d\pm\{\ep_i,\d_p,\ep_i\pm\ep_j,\d_p\pm\d_q,\ep_i\pm\d_p\mid i\neq j,p\neq q\}\\
					&\cup& (2\bbbz+1)\d\pm\{2\ep_i\mid 1\leq i\leq m\}\\
					&\cup& 2\bbbz\d\pm\{2\d_p\mid 1\leq p\leq n\}
				\end{array}$\\
				\hline
				$\stackrel{A(2m-1,2n-1)^{(2)}}{{\hbox{\tiny$(m,n\in\bbbz^{>0},(m,n)\neq (1,1))$}}}$& $\begin{array}{rcl}
					\bbbz\d&\cup& \bbbz\d\pm\{\ep_i\pm\ep_j,\d_p\pm\d_q,\d_p\pm\ep_i\mid i\neq j,p\neq q\}\\
					&\cup& (2\bbbz+1)\d\pm\{2\ep_i\mid 1\leq i\leq m\}\\
					&\cup& 2\bbbz\d\pm\{2\d_p\mid 1\leq p\leq n\}
				\end{array}$\\
				\hline
				$\stackrel{A(2m,2n)^{(4)}}{ {\hbox{\tiny$(m,n\in\bbbz^{\geq0},(m,n)\neq (0,0))$}}}$& $\begin{array}{rcl}
					\bbbz\d&\cup&  \bbbz\d\pm\{\ep_i,\d_p\mid 1\leq i\leq m,\;1\leq p\leq n\}\\
					&\cup& 2\bbbz\d\pm\{\ep_i\pm\ep_j,\d_p\pm\d_q,\d_p\pm\ep_i\mid i\neq j,p\neq q\}\\
					&\cup&(4\bbbz+2)\d\pm\{2\ep_i\mid 1\leq i\leq m\}\\
					&\cup& 4\bbbz\d\pm\{2\d_p\mid 1\leq p\leq n\}
				\end{array}$\\
				\hline
				$\stackrel{D(m+1,n)^{(2)}}{{\hbox{\tiny$(m,n\in\bbbz^{\geq 0},n\neq 0)$}}}$& $\begin{array}{rcl}
					\bbbz\d&\cup&  \bbbz\d\pm\{\ep_i,\d_p\mid 1\leq i\leq m,\;1\leq p\leq n\}\\
					&\cup& 2\bbbz\d\pm\{2\d_p,\ep_i\pm\ep_j,\d_p\pm\d_q,\d_p\pm\ep_i\mid i\neq j,p\neq q\}
				\end{array}$\\
				\hline
		\end{tabular}}
	\end{table}
	
	Here  \[(\d,R)=\{0\},~ (\ep_i,\ep_j)=\d_{i,j},~(\d_p,\d_q)=-\d_{p,q}.\] Moreover, 
	\[\dot R=\{\dot\a\in \sspan_\bbbz\{\ep_i,\d_p\mid i,p\}\mid (\dot\a+\bbbz\d)\cap R\neq \emptyset\}.\] 
	
	We set \begin{equation}  \label{decom-1}
		\begin{array}{ll}
			\hbox{\footnotesize $R_{re} := \{\alpha \in R \mid (\alpha,\alpha) \neq 0\}$}\quad(\hbox{real roots}),&
			\hbox{\footnotesize $R_{im}: =\{\a\in R\mid(\a,R)=\{0\}\}$},\\
			\hbox{\footnotesize $R_{ns} :=\{\a\in R\setminus R_{im}\mid (\alpha,\alpha) = 0\}$}\quad(\hbox{nonsingular roots}),&
			\hbox{\footnotesize $R^\times := R \setminus {R_{im}}.$}
		\end{array}
	\end{equation}
	We have 
	$\dot R^\times=\dot R\setminus\{0\}=\dot R_{ns}\cup \dot R_{re}$ in which $\dot R_{ns}$ and $\dot R_{re} $ are as in the following table:

	\begin{table}[h]\caption{Real and nonsingular roots of $\dot R$} \label{table-ns}
		% "h" means that the table will be placed "here", if we use [t], it will be placed on "top" and [b] indicates the "below". If we do not use this command, the table is floating.
		{\footnotesize \begin{tabular}{|c|l|c|}
				\hline
				$\hbox{Type}$ &\hspace{2.5cm}$\dot R_{re}$ &$\dot R_{ns}$\\
				\hline
				$\stackrel{A(2m,2n-1)^{(2)}}{{\hbox{\tiny$(m,n\in\bbbz^{\geq0},n\neq 0)$}}}$&$\begin{array}{l}
					\{\pm\ep_i,\pm\d_p,\pm2\ep_i,\pm2\d_p,
					\ep_i\pm\ep_j,\d_p\pm\d_q\\
					\mid 1\leq i\neq j\leq m,1\leq p\neq q\leq n\}\\
				\end{array}$& $\begin{array}{l}
					\{\pm \ep_i\pm\d_p\mid 1\leq i\leq m,1\leq p\leq  n\}\\
					\vspace{-2mm}\\
					\sub\dot R_{re}+\dot R_{re}
				\end{array}
				$\\
				\hline
				$\stackrel{A(2m-1,2n-1)^{(2)}}{{\hbox{\tiny$(m,n\in\bbbz^{>0},(m,n)\neq (1,1))$}}}$& $\begin{array}{l}
					\{\pm2\ep_i,\pm2\d_p,
					\ep_i\pm\ep_j,\d_p\pm\d_q\\
					\mid 1\leq i\neq j\leq m,1\leq p\neq q\leq n\}\\
				\end{array}$& $\begin{array}{l}
					\{\pm\ep_i\pm\d_p\mid 1\leq i\leq m,1\leq p\leq  n\}\\
					\vspace{-2mm}\\
					\sub\frac{1}{2}(\dot R_{re}+\dot R_{re})
				\end{array}
				$\\
				\hline
				$\stackrel{A(2m,2n)^{(4)}}{ {\hbox{\tiny$(m,n\in\bbbz^{\geq0},(m,n)\neq (0,0))$}}}$& $\begin{array}{l}
					\{\pm\ep_i,\pm\d_p,\pm2\ep_i,\pm2\d_p,
					\ep_i\pm\ep_j,\d_p\pm\d_q\\
					\mid 1\leq i\neq j\leq m,1\leq p\neq q\leq n\}\\
				\end{array}$ & $\begin{array}{l} \{\pm\ep_i\pm\d_p\mid 1\leq i\leq m,1\leq p\leq  n\}
					\\
					\vspace{-2mm}\\
					\sub\dot R_{re}+\dot R_{re}
				\end{array}
				$
				\\
				\hline
				$\stackrel{D(m+1,n)^{(2)}}{{\hbox{\tiny$(m,n\in\bbbz^{\geq 0},n\neq 0)$}}}$& $\begin{array}{l}
					\{\pm\ep_i,\pm\d_p,\pm2\d_p,
					\ep_i\pm\ep_j,\d_p\pm\d_q\\
					\mid 1\leq i\neq j\leq m,1\leq p\neq q\leq n\}\\
				\end{array}$& $ \begin{array}{l}\{\pm\ep_i\pm\d_p\mid 1\leq i\leq m,1\leq p\leq  n \}
					\\
					\vspace{-2mm}\\
					\sub\dot R_{re}+\dot R_{re}
				\end{array}
				$\\
				\hline
		\end{tabular}}
	\end{table}
	
	\begin{lem}\label{LC}
		The core $\LL_c$ of $\LL$ is a subalgebra of $\LL$ generated by $\cup_{\a\in R^\times}\LL^\a.$ Moreover, The centerless core  $\LL_{cc}$ of $\LL$ is a simple module whose module action is the natural induced action.
	\end{lem}
	\pf For the first assertion, see \cite[Pro.~3.6]{DSY}.  Next assume $I$ is a nonzero ideal of $\LL_c$ with $\bbbc c\sub I.$
	We have the following:
	\begin{itemize}
		\item $\dim(\LL^\a)=1$ for $\a\in R^\times,$
		%\item if $\a\in R_{ns},$ there is $\b\in R_{re}$ with $(\a,\b)\neq 0,$
		\item if $(\a,\b)\neq 0$ and $\LL^\a\sub I,$ then, $\LL^\b\sub I,$
		\item if $\a,\b\in R^\times,$ there are $\a_1,\ldots,\a_k\in R^\times$ with $\a_1=\a,$ $\a_k=\b$ and  $(\a_i,\a_{i+1})\neq 0,$
		\item since $Z(\LL_c)=\bbbc c$, if $ x\in I\cap \LL^\sg\setminus\bbbc c$  $(\sg\in R_{im})$ (see \ref{decom-1}), there is $\a\in R^\times$ with $\{0\}\neq [x,\LL^\a]\sub I\cap \LL^{\a+\sg};$ see \cite[Lem.~3.7]{DSY}.
	\end{itemize}
	These altogether imply that $\Bigop{\a\in  R^\times}{}\LL^\a\sub I$ and so $I=\LL_c$, i.e., $\bbbc c$ is a maximal $\LL$-submodule  of $\LL_c$ and so $\LL_c/\bbbc c$ is a simple $\LL$-module.
	\qed

	\subsubsection{Parabolically induced modules} A subset $P$  of the root system $ R$ of $\LL,$ is called a {\it parabolic subset} of $ R$ if 
	%
	%A subset $P$ of the root system ${{ R}}$ of $ \LL$ is called a {\it parabolic} subset  of ${{ R}}$ if 
	$${ R}=P\cup -P\andd (P+P)\cap { R}\sub P.$$ For a parabolic subset $P$ of $ R,$ we  have the decomposition
	\[ \LL= \LL^+_{_P}\op \LL^\circ_{_P}\op \LL^-_{_P}\] where
	$$\hbox{\small $ \LL^\circ_{_P}:=\Bigop{\a\in P\cap -P}{} \LL^\a,\;  \LL^+_{_P}:=\Bigop{\a\in P\setminus-P}{} \LL^\a\andd  \LL^-_{_P}:=\Bigop{\a\in - P\setminus P}{} \LL^\a.$}$$
	Setting   $\mathfrak{p}:= \LL^\circ_{_P}\op \LL^+_{_P}$, each simple $ \LL^\circ_{_P}$-module $N$ is a module over $\mathfrak{p}$ with the trivial action of $ \LL^+.$ Then  $$\tilde N:=U( \LL)\ot_{U(\mathfrak{p})}N$$ is an $ \LL$-module; here  $U( \LL)$ and $U(\mathfrak{p})$ denote, respectively,  the universal enveloping algebras of $ \LL$ and $\mathfrak{p}.$ If the  $ \LL$-module  $\tilde N$ contains a  unique maximal submodule $Z$ {intersecting} $N$ trivially, the quotient module $${\rm Ind}^{ \LL}_{P}(N):=\tilde N/Z$$ is called  a {\it parabolically induced} module if $P\neq R$. A simple  $ \LL$-module  which is not parabolically induced is called {\it cuspidal}.
	
	Assume $\lam$ is a functional on the $\bbbr$-linear span of $R$. Then, 
	$$P_\lam:=R^+\cup R^\circ$$
	 is a parabolic subset of $R.$ We denote $\LL_{_{P_\lam}}^\circ$ and $\LL_{_{P_\lam}}^\pm$  
	by $\LL_{_{\lam}}^\circ$ and $\LL_{_{\lam}}^\pm$  respectively and ${\rm Ind}^{ \LL}_{P}(N)$ by ${\rm Ind}^{ \LL}_{\lam}(N)$.
	
	\begin{deft}
		{\rm Suppose that $\fk$ is a subalgebra of $\LL$ containing $\hh.$ A superspace  ${V}={V}_0\op {V}_1$ is called  an {\it {$\cH$}-weight  $\fk$-module} (or simply a weight $\fk$-module if there is no ambiguity) if
			\begin{itemize}
				\item [(1)] $[x,y]v=x(yv)-(-1)^{\mid x\mid\mid y\mid}y(xv)$ for all $v\in {V}$ and  $x,y\in\fk$,
				\item [(2)] $\fk_i {V}_j\sub {V}_{i+j }$ for all $i,j\in\{0,1\}$, 
				\item[(3)] ${V}=\op_{\lam\in \cH^\ast}{V}^\lam$ with
				${V}^\lam:=\{v\in {V}\mid hv=\lam(h)v\;\;(h\in\cH)\}$ for each $\lam\in \cH^\ast.$
			\end{itemize}
			In this setting, an element $\lam$ of the {\it support} of $V$, defined by
			$$\supp({V}):=\{\lam\in \fh^*\mid {V}^\lam\neq \{0\}\},$$
			is called a {\it weight} of $V$, and the corresponding ${V}^\lam$ is called a {\it weight space}. Elements of a weight space are called {\it weight vectors}. If all weight spaces are  finite dimensional, then the module ${V}$ is called a {\it finite weight module}.
		}
	\end{deft}

	The following proposition is known in the literature: 
	\begin{Pro}\label{ind}
		Suppose that $\lam$ is a functional on the $\bbbr$-linear span of $R$ and set $P:=P_{\lam}$. If $V$ is a simple  finite $\hh$-weight $ \LL$-module with $$V^{ \LL_{_\lam}^+}:=\{v\in V\mid  \LL_{_\lam}^+ v=\{0\}\}\neq \{0\},$$ then $V^{ \LL_{_\lam}^+}$ is a simple finite $\hh$-weight $ \LL_{_\lam}^\circ$-module and  $V\simeq {\rm Ind}_{\lam}^ \LL(V^{ \LL_{_\lam}^+}).$
	\end{Pro}

	\subsubsection{${A(2m,2n)^{(4)}}$}\label{type 2m-2n}
	Suppose that $m\in\bbbz^{\geq 1}$ and $n\in\bbbz^{\geq 0}.$ For $1\leq p\leq 2n+1,$ set $\bar p:=p+2m+1$ and suppose $$\mathscr{G}:=\left\{
	\begin{array}{ll}
		\frak{sl}(2m+1,2n+1)& m\neq n\\
		\frak{psl}(2m+1,2n+1)& m=n.
	\end{array}
	\right.$$
	Set
	\[h_i:=e_{i,i}-e_{i+1,i+1},\; d_p:=e_{\bar p,\bar p}-e_{\overline{p+1},\overline{p+1}}, \jj:=(1-\d_{m,n})(e_{m+1,m+1}+e_{\overline{n+1},\overline{n+1}})\] for $1\leq i\leq 2m$ and $1\leq p\leq 2n$.  Then
	\[{H}:=\bbbc \jj\op\sum_{i=1}^{2m}\bbbc h_i\op\sum_{p=1}^{2n}\bbbc d_p\] is the standard Cartan subalgebra of $\mathscr{G}$ with the set of roots
	\[\dot{\mathfrak{s}}:=\{\dot\ep_i-\dot\ep_j,\dot\d_p-\dot\d_q,\pm(\dot\ep_i-\dot\d_p)\mid 1\leq i,j\leq2m+1,1\leq p,q\leq 2n+1 \}\]  where for $1\leq i\leq 2m,$  $1\leq j\leq 2m+1,$  $1\leq p\leq 2n$ and $1\leq q\leq 2n+1,$
	\[\begin{array}{llcll}
		\dot\ep_j:&{H}\longrightarrow \bbbc&\hbox{and}&\dot\d_q:&{H}\longrightarrow \bbbc\\
		&d_p\mapsto 0,\;h_i\mapsto \d_{i,j}-\d_{i+1,j},&&&h_i\mapsto 0,\;d_p\mapsto \d_{q,p}-\d_{p+1,q},\\
		&\jj\mapsto (1-\d_{m,n})\d_{j,m+1},&&&\jj\mapsto (1-\d_{m,n})\d_{q,n+1},\\
	\end{array}
	\] For an integer $s,$ define  \[sgn(s):=\left\{
	\begin{array}{ll}
		1&s>0\\
		0&s\leq 0.
	\end{array}
	\right.
	\]
	Now assume $\zeta$ is the $4$-th primitive root of unity and define the automorphism $\sg$ of $\mathscr{G}$  such that for $1\leq i\neq j\leq 2m+1$ and $1\leq p\neq  q\leq 2n+1,$
	{\begin{align*}
			e_{i,j}\mapsto&-(-1)^{i+j}e_{2m+2-j,2m+2-i},\\
			e_{\bar p,\bar q}\mapsto&-(-1)^{p+q+sgn( n+1-p)+sgn(n+1-q)+(n+1)(\d_{p,n+1}+\d_{q,n+1})}\zeta^{\d_{q,n+1}}(-\zeta)^{\d_{p,n+1}}e_{\overline{2n+2-q},\overline{2n+2-p}},\\
			e_{i,\bar p}\mapsto&-(-1)^{i+p}(-1)^{sgn(n+1-p)}(-1)^{(n+1)\d_{n+1,p}}\zeta^{\d_{n+1,p}}e_{\overline{2n+2-p},2m+2-i}, \\
			e_{\bar p,i}\mapsto&(-1)^{i+p}(-1)^{sgn(n+1-p)}(-1)^{(n+1)\d_{n+1,p}}(-\zeta)^{\d_{n+1,p}}e_{2m+2-i,\overline{2n+2-p}},\\
			h_i\mapsto& h_{2m+1-i}\quad(i\neq 2m+1),\\
			d_p\mapsto& d_{2n+1-p}\quad(p\neq 2n+1),\\
			\jj\mapsto& -\jj.
	\end{align*}}
	For $1\leq i\leq 2m+1$ and $1\leq p\leq 2n+1,$ we have
	$\sg(\dot\ep_i)=-\dot \ep_{2m+2-i}$ and $\sg(\dot\d_p)=-\dot \d_{2n+2-p}.$ 
	Moreover, the fixed point subalgebra  of $H$  under $\sg$ is 
	\[\fh=\sum_{r=1}^m\bbbc(h_r+h_{2m+1-r})+\sum_{p=1}^{n}\bbbc(d_p+d_{2n+1-p})\in \sspan_\bbbc\{e_{i,i},e_{\bar p,\bar p}\mid i\neq m+1,j\neq n+1\}.\]
	For \[
	\ep_i:=\pi(\dot\ep_i)=\dot\ep_i\mid_{\fh}\andd \d_p:=\pi(\dot\d_p)=\dot\d_p\mid_{\fh},\] the root system of $\mathscr{G}$ with respect to $\fh$ is 
	\[\dot R=\pi(\dot{\frak{s}})=\{\pm\ep_i,\pm\ep_i\pm\ep_j,\pm\d_p,\pm\d_p\pm\d_q,\pm\ep_i\pm\d_p\mid 1\leq i,j\leq m,\;1\leq p,q\leq n\}.\] Assume  $1\leq r\neq s\leq m$ and $1\leq p\neq q\leq n.$ In what follows we give the expression of ${}^{[k]}\mathscr{G}^{(\a)}$'s: 
	
	\medskip
	$\bs{{}^{[0]}\mathscr{G}^{(\a)}~~ (\a\neq 0)}:$
	{\footnotesize \[\left\{
		\begin{array}{ll}
			\bbbc(e_{r,s}-(-1)^{r+s}e_{2m+2-s,2m+2-r})& \a=\ep_r-\ep_s ,\\\\
			\bbbc(e_{r,2m+2-s}-(-1)^{r+s}e_{s,2m+2-r})& \a=\ep_r+\ep_s,\\\\
			\bbbc(e_{2m+2-r,s}-(-1)^{r+s}e_{2m+2-s,r})& \a=-\ep_r-\ep_s,\\\\
			\bbbc(e_{r,m+1}+(-1)^{r+m}e_{m+1,2m+2-r})& \a=\ep_r,\\\\
			\bbbc(e_{m+1,r}+(-1)^{r+m}e_{2m+2-r,m+1})& \a=-\ep_r,\\\\
			\bbbc(e_{\bar p,\bar q}-(-1)^{ p+ q}e_{\overline{2n+2-q},\overline{2n+2-p}})& \a=\d_p-\d_q,\\\\
			\bbbc(e_{\bar p,\overline{2n+2- q}}+(-1)^{ p+ q}e_{\overline{q},\overline{2n+2-p}})& \a=\d_p+\d_q,\\\\
			\bbbc(e_{\overline{2n+2- p},\bar q}+(-1)^{ p+ q}e_{\overline{2n+2-q},\overline{p}})& \a=-\d_p-\d_q,\\\\
		\end{array}\right.\andd
		\left\{
		\begin{array}{ll}
			\bbbc e_{\bar p,\overline{2n+2- p}}& \a=2\d_p,\\\\
			\bbbc e_{\overline{2n+2- p},\bar p}& \a=-2\d_p,\\\\
			\bbbc(e_{\bar p,m+1}+(-1)^{m+p}e_{m+1,\overline{2n+2-p}})&\a=\d_p,\\\\
			\bbbc(e_{m+1,\bar p}-(-1)^{m+p}e_{\overline{2n+2-p},m+1})&\a=-\d_p,\\\\
			\bbbc(e_{r,\overline{ p}}+(-1)^{r+p}e_{\overline{2n+2-p},2m+2-r})&\a=\ep_r-\d_p,\\\\
			\bbbc(e_{\overline{ p},r}-(-1)^{r+p}e_{2m+2-r,\overline{2n+2-p}})&\a=\d_p-\ep_r,\\\\
			\bbbc(e_{r,\overline{2n+2- p}}-(-1)^{r+p}e_{\overline{p},2m+2-r})&\a=\ep_r+\d_p,\\\\
			\bbbc(e_{2m+2-r,\bar p}+(-1)^{r+p}e_{\overline{2n+2-p},r})&\a=-\ep_r-\d_p.
		\end{array}
		\right.
		\]}
	$\bs{{}^{[2]}\mathscr{G}^{(\a)}~~ (\a\neq 0)}:$
	{\footnotesize \[\left\{
		\begin{array}{ll}
			\bbbc(e_{r,s}+(-1)^{r+s}e_{2m+2-s,2m+2-r})\quad & \a=\ep_r-\ep_s,\\\\
			\bbbc e_{r,2m+2-r}& \a=2\ep_r,\\\\
			\bbbc e_{2m+2-r,r}& \a=-2\ep_r,\\\\
			\bbbc(e_{r,2m+2-s}+(-1)^{r+s}e_{s,2m+2-r})& \a=\ep_r+\ep_s,\\\\
			\bbbc(e_{2m+2-r,s}+(-1)^{r+s}e_{2m+2-s,r})& \a=-\ep_r-\ep_s,\\\\
			\bbbc(e_{r,m+1}-(-1)^{r+m}e_{m+1,2m+2-r})& \a=\ep_r,\\\\
			\bbbc(e_{m+1,r}-(-1)^{r+m}e_{2m+2-r,m+1})& \a=-\ep_r,\\\\
			\bbbc(e_{\bar p,\bar q}+(-1)^{ p+ q}e_{\overline{2n+2-q},\overline{2n+2-p}})\quad & \a=\d_p-\d_q,\\\\
		\end{array}
		\right.\andd 
		\left\{
		\begin{array}{ll}
			\bbbc(e_{\bar p,\overline{2n+2- q}}-(-1)^{ p+ q}e_{\overline{q},\overline{2n+2-p}})\quad & \a=\d_p+\d_q,\\\\
			\bbbc(e_{\overline{2n+2- p},\bar q}-(-1)^{ p+ q}e_{\overline{2n+2-q},\overline{p}})\quad & \a=-\d_p-\d_q,\\\\
			\bbbc(e_{\bar p,m+1}-(-1)^{m+p}e_{m+1,\overline{2n+2-p}})&\a=\d_p,\\\\
			\bbbc(e_{m+1,\bar p}+(-1)^{m+p}e_{\overline{2n+2-p},m+1})&\a=-\d_p,\\\\
			\bbbc(e_{r,\overline{ p}}-(-1)^{r+p}e_{\overline{2n+2-p},2m+2-r})&\a=\ep_r-\d_p,\\\\
			\bbbc(e_{\overline{ p},r}+(-1)^{r+p}e_{2m+2-r,\overline{2n+2-p}})&\a=\d_p-\ep_r,\\\\
			\bbbc(e_{r,\overline{2n+2- p}}+(-1)^{r+p}e_{\overline{p},2m+2-r})&\a=\ep_r+\d_p,\\\\
			\bbbc(e_{2m+2-r,\bar p}-(-1)^{r+p}e_{\overline{2n+2-p},r})&\a=-\ep_r-\d_p.
		\end{array}
		\right.
		\]}
	
	\[
	\begin{array}{lll}
		\bs{{}^{[1]}\mathscr{G}^{(\a)}~~ (\a\neq 0)}:&&\bs{{}^{[3]}\mathscr{G}^{(\a)}~~ (\a\neq 0)}:\\\\
		{\footnotesize \left\{
			\begin{array}{ll}
				\bbbc(e_{r,\overline{n+1}}-(-1)^{r}e_{\overline{n+1},2m+2-r})& \a=\ep_r,\\\\
				\bbbc(e_{2m+2-r,\overline{n+1}}-(-1)^{r}e_{\overline{n+1},r})& \a=-\ep_r,\\\\
				\bbbc(e_{\bar p,\overline{n+1}}+(-1)^{p}e_{\overline{n+1},\overline{2n+2-p}})&\a=\d_p,\\\\
				\bbbc(e_{\overline{n+1},\bar p}-(-1)^{p}e_{\overline{2n+2-p},\overline{n+1}})&\a=-\d_p.
			\end{array}
			\right.}
		&&
		{\footnotesize  \left\{
			\begin{array}{ll}
				\bbbc(e_{r,\overline{n+1}}+(-1)^{r}e_{\overline{n+1},2m+2-r})& \a=\ep_r,\\\\
				\bbbc(e_{2m+2-r,\overline{n+1}}+(-1)^{r}e_{\overline{n+1},r})&\a=-\ep_r\\\\
				\bbbc(e_{\bar p,\overline{n+1}}-(-1)^{p}e_{\overline{n+1},\overline{2n+2-p}})&\a=\d_p,\\\\
				\bbbc(e_{\overline{n+1},\bar p}+(-1)^{p}e_{\overline{2n+2-p},\overline{n+1}})&\a=-\d_p.
			\end{array}
			\right.}
	\end{array}\]
	
	Also, we have 
	{\footnotesize\[{}^{[k]}\mathscr{G}^{(0)}=\left\{\begin{array}{ll}
			\bbbc(e_{m+1,\overline{n+1}}+(-1)^me_{\overline{n+1},m+1})&k=1,\\\\
			\bbbc(e_{m+1,\overline{n+1}}-(-1)^me_{\overline{n+1},m+1})&k=3,\\\\
			\displaystyle{\sum_{r=1}^m\bbbc(h_r+h_{2m+1-r})+\sum_{p=1}^{n}\bbbc(d_p+d_{2n+1-p})} & k=0,\\\\
			\displaystyle{\sum_{r=1}^m\bbbc(h_r-h_{2m+1-r})+\sum_{p=1}^{n}\bbbc(d_p-d_{2n+1-p})+\bbbc\jj} =&k=2\\
			\displaystyle{\sum_{r=1}^{m-1}\bbbc(h_r-h_{2m+1-r})+\sum_{p=1}^{n-1}\bbbc(d_p-d_{2n+1-p})+\bbbc x+\bbbc y+\bbbc(1-\d_{m,n})\ii,} &
		\end{array}
		\right.
		\]}
	where 
	\[\ii=\frac{1}{2m+1}(\sum_{i=1}^{2m+1}e_{i,i})+\frac{1}{2n+1}(\sum_{i=1}^{2n+1}e_{\bar i,\bar i}),\] and 
	\begin{align*}
		x:=&\frac{1}{2}(\frac{-1}{2}(e_{m,m}+e_{m+2,m+2})+\frac{1}{2}(e_{\bar n,\bar n}+e_{\overline{n+2},\overline{n+2}})+e_{m+1,m+1}-e_{\overline{n+1},\overline{n+1}})\\
		=&\frac{1}{2}(\frac{1}{2}(d_n-d_{n+1})-\frac{1}{2}(h_m-h_{m+1})),\\
		y:=&2(e_{m+1,m+1}+e_{\overline{n+1},\overline{n+1}})\quad({\rm mod} ~\d_{m,n}\ii)\\
		%=&2((1-\d_{m,n})\ii+\frac{1}{2m+1}\sum_{i=1}^{2m+1}(e_{m+1,m+1}-e_{i,i})+\frac{1}{2n+1}\sum_{i=1}^{2n+1}(e_{\overline{n+1},\overline{n+1}}-e_{\bar i,\bar i}))\\
		%=&2((1-\d_{m,n})\ii+\frac{1}{2m+1}(\sum_{i=1}^{m}(e_{m+1,m+1}-e_{i,i})+\sum_{i=m+2}^{2m+1}(e_{m+1,m+1}-e_{i,i}))\\
		%+&\frac{1}{2n+1}(\sum_{i=1}^{n}(e_{\overline{n+1},\overline{n+1}}-e_{\bar i,\bar i})+\sum_{i=n+2}^{2n+1}(e_{\overline{n+1},\overline{n+1}}-e_{\bar i,\bar i})))\\
		%=&\red{2}((1-\d_{m,n})\ii+\frac{1}{2m+1}(\sum_{i=1}^{m}(e_{m+1,m+1}-e_{i,i})+\sum_{i=1}^{m}(e_{m+1,m+1}-e_{2m+2-i,2m+2-i}))\\
		%+&\frac{1}{2n+1}(\sum_{i=1}^{n}(e_{\overline{n+1},\overline{n+1}}-e_{\bar i,\bar i})+\sum_{i=n+2}^{2n+1}(e_{\overline{n+1},\overline{n+1}}-e_{\overline{2n+2-i},\overline{2n+2-i}})))\\
		%=&\red{2}((1-\d_{m,n})\ii+\frac{1}{2m+1}(\sum_{i=1}^{m}\sum_{j=i}^{m}-(e_{j,j}-e_{j+1,j+1})+\sum_{i=1}^{m}\sum_{j=i}^{m}(e_{2m+1-j,2m+1-j}-e_{2m+2-j,2m+2-j}))\\
		%+&\frac{1}{2n+1}(\sum_{i=1}^{n}\sum_{j=i}^{n}-(e_{\bar j,\bar j}-e_{\overline{j+1},\overline{j+1}})+\sum_{i=1}^{n}\sum_{j=i}^{n}(e_{\overline{2n+1-j},\overline{2n+1-j}}-e_{\overline{2n+2-j},\overline{2n+2-j}})))\\
		=&2((1-\d_{m,n})\ii-\frac{1}{2m+1}(\sum_{i=1}^{m}\sum_{j=i}^{m}(h_j-h_{2m+1-j})
		-\frac{1}{2n+1}(\sum_{i=1}^{n}\sum_{j=i}^{n}(d_{j}-d_{{2n+1-j}}).
	\end{align*}
	Recall that  $\zeta$ is the forth primitive root  of unity and set 
	
	\begin{align*}
		e:={\zeta^m}(e_{m+1,\overline{n+1}}+(-1)^me_{\overline{n+1},m+1}), \quad
		f:={\zeta^m}(e_{m+1,\overline{n+1}}-(-1)^me_{\overline{n+1},m+1}),\\
	\end{align*}
	and 
	\begin{equation*}\label{exy}
		\begin{array}{ll}
			e_{4k+1}:=e\ot t^{4k+1}\in\LL^{(4k+1)\d}\sub\LL_1,&
			e_{4k-1}:=f\ot t^{4k-1}\in\LL^{(4k-1)\d}\sub\LL_1,\\
			x_{4k+2}:=x\ot t^{4k+2}\in \LL^{(4k+2)\d}\sub\LL_0,&
			y_{4k+2}:=y\ot t^{4k+2}\in \LL^{(4k+2)\d}\sub\LL_0.
		\end{array}
	\end{equation*}
	We have 
	\begin{equation*}\label{ef}
		\Bigop{k\in\bbbz}{}\LL^{(2k+1)\d}=\Bigop{k\in\bbbz}{}\bbbc e_{2k+1}\hbox{ with }\LL^{(4k+1)\d}=e\ot t^{4k+1}\andd \LL^{(4k-1)\d}=f\ot t^{4k-1}
	\end{equation*}  and 
	\[\Bigop{k\in\bbbz}{}\LL^{2k\d}=\LL^0\op\mathscr{K}\op\Bigop{k\in\bbbz}{}\bbbc x_{4k+2}\op\Bigop{k\in\bbbz}{}\bbbc y_{4k+2}\op((1-\d_{m,n})\ii)\ot t^2\bbbc[t^{\pm4}])\] where 
	\begin{comment}
		We have for $1\leq r\leq m-1,$ up to $\bbbc c$ that 
		\begin{align*}
			&[(e_{r,r+1}+e_{2m+1-r,2m+2-r})\ot t^{4k},(e_{r+1,r}-e_{2m+2-r,2m+1-r})\ot t^{2}]\\
			=&(e_{r,r}-e_{r+1,r+1}-e_{2m+1-r,2m+1-r}+e_{2m+2-r,2m+2-r})\ot t^{4k+2}\\
			=&(h_r-h_{2m+1-r})\ot t^{4k+2}
		\end{align*}

		\begin{align*}
			[e_{r,2m+2-r}\ot t^{2},e_{2m+2-r,r}\ot t^{4k-2}]=&(e_{r,r}-e_{2m+2-r,2m+2-r})\ot t^{4k}\\
			=&(e_{r,r}-e_{2m+2-r,2m+2-r})\ot t^{4k}\\
			=&\sum_{i=r}^m(h_i+h_{2m+1-i})\ot t^{4k}\\
		\end{align*}
		and if $r\neq m$
		\begin{align*}
			(h_r+h_{2m+1-r})\ot t^{4k}=&(e_{r,r}-e_{r+1,r+1}+e_{2m+1-r,2m+1-r}-e_{2m+2-r,2m+2-r})\ot t^{4k}\\
			=&(e_{r,r}-e_{2m+2-r,2m+2-r})\ot t^{4k}-(e_{r+1,r+1}-e_{2m+1-r,2m+1-r})\ot t^{4k}\\
			=&[e_{r,2m+2-r}\ot t^{2},e_{2m+2-r,r}\ot t^{4k-2}]+[e_{r+1,2m+1-r}\ot t^{2},e_{2m+1-r,r+1}\ot t^{4k-2}]
		\end{align*}
		and for $r=m,$
		\begin{align*}
			(h_r+h_{2m+1-r})\ot t^{4k}=&(e_{m,m}-e_{m+1,m+1}+e_{m+1,m+1}-e_{m+2,m+2})\ot t^{4k}\\
			=&(e_{m,m}-e_{m+2,m+2})\ot t^{4k}\\
			=&[e_{m,m+2}\ot t^{2},e_{m+2,m}\ot t^{4k-2}]
		\end{align*}
	\end{comment}
	\begin{align*}\label{k-final}
		\mathscr{K}:=&
		\sum_{\substack{k,k'\in\bbbz,k+k'\neq 0\nonumber\\
				\dot\a\in \dot R\setminus \{0,\pm\ep_i,\pm\d_p\mid i,p\}}}[\LL^{\dot\a+k\d},\LL^{-\dot\a+k'\d}]\\
		=&(\displaystyle{\sum_{r=1}^{m-1}\bbbc(h_r-h_{2m+1-r})+\sum_{p=1}^{n-1}\bbbc(d_p-d_{2n+1-p})}){{\otimes t^2\bbbc[t^{\pm4}]}}\\
		+&((\displaystyle{\sum_{r=1}^{m}\bbbc(h_r+h_{2m+1-r})+\sum_{p=1}^{n}\bbbc(d_p+d_{2n+1-p})}){{\otimes \bbbc[t^{\pm4}])\setminus \LL^0}}\nonumber
	\end{align*}  satisfy
	\[[\mathscr{K},\mathscr{K}]=\{0\}\andd [\mathscr{K},\Bigop{k\in\bbbz}{}\LL^{(2k+1)\d}]=\{0\}\quad (\hbox{modulo $\bbbc c$}).\]
	We also have, up to $\bbbc c,$ that 
	\begin{equation*}\label{bracket}
		\begin{array}{ll}
			~[e_{4k+1},e_{4k'+1}]= y_{_{4(k+k')+2}},&
			~[e_{4k-1},e_{4k'-1}]=-y_{_{4(k+k')-2}},\\
			~[x_{4k'-2},e_{4k+1}]= e_{_{4(k+k')-1}},&
			~[x_{4k'+2},e_{4k-1}]=e_{_{4(k+k')+1}},\\
			~[y_{4k'-2},e_{4k+1}]=0,&
			~[y_{4k'+2},e_{4k-1}]=0,\\
			~[e_{4k'-1},e_{4k+1}]=0,&~[y_{4k'+2},x_{4k+2}]=0,\\
			~[x_{4k'+2},x_{4k+2}]=0,&~[x_{4k'+2},x_{4k+2}]=0.
		\end{array}
	\end{equation*}
	In  fact, setting  
	\[\aa:=((1-\d_{m,n})\ii)\ot t^2\bbbc[t^{\pm2}])
	\op\mathscr{K}\]
	which is a $\bbbz$-graded abelian Lie algebra  and 
	\[\mathcal{N}:=\Bigop{k\in\bbbz}{}\bbbc x_{4k+2}\op\Bigop{k\in\bbbz}{}\bbbc y_{4k+2}
	\Bigop{k\in\bbbz}{}\bbbc e_{4k+1}\op\Bigop{k\in\bbbz}{}\bbbc e_{4k-1}\] which is a  $\bbbz$-graded non-abelian Lie superalgebra, we have 
	$\fl:=\Bigop{k\in\bbbz}{}\LL^{k\d}\simeq (\aa\op\mathcal{N})\rtimes \hh$.
	We can identify $\cN$ with the quadratic Lie superalgebra $\cQ,$ see \S~\ref{quadratic}.
	
	\subsubsection{Zero-level finite weight modules over $\displaystyle{\bigoplus_{k\in\bbbz}\LL^{k\d}}$} \label{zero-level-M} Assume $\LL$ is a twisted affine Lie superalgebra whether  of type $X=A(2m,2n)^{(4)}$ or $X\not=A(2m,2n)^{(4)}.$
	Set
	\[\mathscr{L}:=\Bigop{0\neq k\in\bbbz}{}\LL^{k\d}\andd \fl:=\Bigop{k\in\bbbz}{}\LL^{k\d}.\] Then, $\mathscr{L}$ is a subalgebra of the centerless core $\LL_{cc}$ of $\LL$ and 
	\[\fl=\mathscr{L}\rtimes\hh.\]
	For 
	a functional $\lam$ on $\fh\op\bbbc d$ and a finite weight $\bbbz$-graded $\mathscr{L}$-module $\Omega,$   set $\Omega(\lam):=\Omega$ and define 
	\begin{align*}
		\cdot:&\fl\times \Omega(\lam)\longrightarrow \Omega(\lam)\\
		&(x,v)\mapsto\left\{
		\begin{array}{ll}
			xv& \hbox{if $x\in\mathscr{L}$}\\
			(\lam(h+rd)+rk)v& \hbox{if $v\in \Omega^k$ {{$(k\in\bbbz)$}},  $x=h+sc+rd$ ($h\in\fh,~r,s\in\bbbc$).}
		\end{array}
		\right.
	\end{align*}
	Consider $\lam$ as a functional on $\hh$ with $\lam(c)=0.$ Then, $\Omega(\lam)$ is an $\hh$-weight  $\fl$-module of level zero 
	with \[\supp(\Omega(\lam))=\lam+{{\supp_\bbbz(\Omega)\d}}\] and 
	for  $\lam+k\d\in\supp(\Omega(\lam)),$ we have  $\Omega(\lam)^{\lam+k\d}=\Omega^k$ for $k\in\bbbz.$ 
	Moreover, $\Omega(\lambda)$ is a finite (resp. simple) $\hh$-weight  $\fl$-module of level zero if and only if $\Omega$ is a  finite weight {{$\bbbz$-graded}} (resp. $\bbbz$-graded-simple) $\mathscr{L}$-module.
	
	Conversely, assume $N$ is a zero-level simple  finite $\hh$-weight  $\mathfrak{L}$-module. Fix $\lam\in \supp(N).$ Since the level of $N$ is zero, $\lam(c)=0$ and so, we identify $\lam$ with a functional on $\fh\op\bbbc d$. Also, as $N$ is simple, $\supp(N)\sub \lam+\bbbz\d.$ Then, $\Omega:= N$ is a $\bbbz$-graded $\mathscr{L}$-module with 
	\[\Omega^k=N^{\lam+k\d} \quad(k\in\bbbz).\] Moreover, $\Omega$ is a finite weight $\bbbz$-graded simple $\mathscr{L}$-module and $N=\Omega(\lam).$
	
	\section{Zero-level integrable finite weight modules}
	Suppose that $\LL=\LL_0\oplus\LL_1$ $(\bbbz_2=\{1,2\})$  is a twisted  affine Lie superalgebra of type $X= A(2m,2n-1)^{(2)}, A(2m-1,2n-1)^{(2)}$ $((m,n)\neq (1,1)), A(2m,2n)^{(4)}$ and $D(m+1,n)^{(2)}$ with standard Cartan subalgebra  $\hh$. Recall Table~\ref{table1} for the root system $R=R_0\cup R_1.$		
	As we see in this table, there is a positive integer {{$p_*=2,4$}} such that 
	\begin{equation}\label{p}
		R+p_*\bbbz\d\sub R.
	\end{equation}

	Suppose that $V$ is an $\LL$-module. We denote by  ${R}^{ln}({V})$,   the set of all real roots $\a\in R$  for which $0\neq x\in\LL^\a$ acts on ${V}$ locally nilpotently.
	It is a well-known fact that for a  finite weight $\LL$-module $V=\op_{\lam\in \cH^*}V^\lam,$			
	\begin{equation}\label{in-supp 1}
		\parbox{6in}{if $\mu\in\supp(V)$ and $\pm\a\in R^{ln}(V)$, then $2(\mu,\a)/(\a,\a)\in\bbbz$; moreover if $2(\mu,\a)/(\a,\a)\in\bbbz^{>0},$ then $\mu-\a\in \supp(V).$}
	\end{equation}
	
	\begin{deft}
		An $\LL$-module $V$ is called {\it integrable} if
		\begin{itemize}
			\item[(1)] 
			$V$ is a finite $\cH$-weight $\LL$-module,
			\item[(2)] 
			$R_{re}=R^{ln}(V).$
		\end{itemize}
		It is called {\it admissible} if the weight multiplicities are uniformly bounded.
	\end{deft}

	The classifications of simple finite weight  modules over an affine Lie superalgebra strongly depends on the fact that if the level is zero or not; e.g., although, it has been proven in \cite{E1}, \cite{Y1} that there is no nonzero-level simple  integrable finite weight $\LL$-module for almost all affine Lie superalgebras with nonzero odd part, using Lemma~\ref{LC}, we get that  the centerless core of $\LL$ is a simple  integrable finite weight $\LL$-modules of level-zero. 
	
	\medskip
	
	In this section, we  characterize simple integrable finite weight $\LL$-modules.  We know form \cite[Pro. 2.5]{Y2} that for the root system $R_0$ of $\LL_0$ with respect to $\hh$,  $R_0\cap R_{re}=\{\a\in R_{re}\mid 2\a\not\in R_{re}\}$.
	Set $\dot S:=\dot{R}\cap R_0$, that is 
	
	\bigskip
	
	\begin{table}[h]\caption{Finite root system $\dot{S}$} \label{table11}
		% "h" means that the table will be placed "here", if we use [t], it will be placed on "top" and [b] indicates the "below". If we do not use this command, the table is floating.
		{\footnotesize \begin{tabular}{|c|l|}
				\hline
				$\hbox{Type}$ &\hspace{3.25cm}$\dot{S}$ \\
				\hline
				$\stackrel{A(2m,2n-1)^{(2)}}{{\hbox{\tiny$(m,n\in\bbbz^{\geq0},n\neq 0)$}}}$&$
				\pm\{0,\ep_i,\ep_i\pm\ep_j,2\d_p, \d_p\pm\d_q\mid 1\leq i\neq j\leq m,1\leq p\neq q\leq n\}$\\
				\hline
				$\stackrel{A(2m-1,2n-1)^{(2)}}{{\hbox{\tiny$(m,n\in\bbbz^{>0},(m,n)\neq (1,1))$}}}$& $
				\pm\{0,  \ep_i\pm\ep_j, 2\d_p, \d_p\pm\d_q\mid 1\leq i\neq j\leq m,1\leq p\neq q\leq n\}$\\
				\hline
				$\stackrel{A(2m,2n)^{(4)}}{ {\hbox{\tiny$(m,n\in\bbbz^{\geq0},(m,n)\neq (0,0))$}}}$& $
				\pm\{0, \e_i,  \ep_i\pm\ep_j,2\d_p,\d_p\pm\d_q\mid 1\leq i\neq j\leq m,1\leq p\neq q\leq n\}	
				$\\
				\hline
				$\stackrel{D(m+1,n)^{(2)}}{{\hbox{\tiny$(m,n\in\bbbz^{\geq 0},l\neq 0)$}}}$& $
				\pm\{0,\ep_i, \ep_i\pm\ep_j,2\d_p,\d_p\pm\d_q\mid 1\leq i\neq j\leq m,1\leq p\neq q\leq n\}
				$\\
				\hline
		\end{tabular}}
	\end{table}

	We point it out that {{$\fk:=\op_{\dot\a\in \dot S}\LL^{\dot\a}$}} is a reductive Lie algebra. We fix bases $\dot \D$ and $\dot \Pi$ respectively for $\dot S$ and $\dot R_{re}$ as in the following table:
	
	{\footnotesize 
		\begin{table}[h]
			\begin{tabular}{|c|l|l|}
				\hline
				$\hbox{Type}$ &\hspace{3.25cm}$\dot{\D}$ & \hspace{3cm} $\dot \Pi$ \\
				\hline
				$\stackrel{A(2m,2n-1)^{(2)}}{{\hbox{\tiny$(m,n\in\bbbz^{\geq0},l\neq 0)$}}}$&
				$\{\ep_i-\ep_{i+1},\ep_{m}\}_{i=1}^{m-1}\cup\{\d_p-\d_{p+1},2\d_n\}_{p=1}^{n-1}$&$\{\ep_i-\ep_{i+1},\ep_{m}\}_{i=1}^{m-1}\cup\{\d_p-\d_{p+1},\d_n\}_{p=1}^{n-1}$\\
				\hline
				$\stackrel{A(2m-1,2n-1)^{(2)}}{{\hbox{\tiny$(m,n\in\bbbz^{>0},(m,n)\neq (1,1))$}}}$& $\{\ep_i-\ep_{i+1},\ep_{m-1}+\ep_{m}\}_{i=1}^{m-1}\cup\{\d_p-\d_{p+1},2\d_n\}_{p=1}^{n-1}$&
				$\{\ep_i-\ep_{i+1},2\ep_{m}\}_{i=1}^{m-1}\cup\{\d_p-\d_{p+1},2\d_n\}_{p=1}^{n-1}$\\
				\hline
				$\stackrel{A(2m,2n)^{(4)}}{ {\hbox{\tiny$(m,n\in\bbbz^{\geq0},(m,n)\neq (0,0))$}}}$&
				$\{\ep_i-\ep_{i+1},\ep_{m}\}_{i=1}^{m-1}\cup\{\d_p-\d_{p+1},2\d_n\}_{p=1}^{n-1}$&$\{\ep_i-\ep_{i+1},\ep_{m}\}_{i=1}^{m-1}\cup\{\d_p-\d_{p+1},\d_n\}_{p=1}^{n-1}$\\
				\hline
				$\stackrel{D(m+1,n)^{(2)}}{{\hbox{\tiny$(m,n\in\bbbz^{\geq 0},n\neq 0)$}}}$& 
				$\{\ep_i-\ep_{i+1},\ep_{m}\}_{i=1}^{m-1}\cup\{\d_p-\d_{p+1},2\d_n\}_{p=1}^{n-1}$&$\{\ep_i-\ep_{i+1},\ep_{m}\}_{i=1}^{m-1}\cup\{\d_p-\d_{p+1},\d_n\}_{p=1}^{n-1}$\\
				\hline
			\end{tabular}
		\end{table}
	}

	Using the same argument as in \cite[Lem. 2.3(a)]{Y1}, we have the following lemma:
	\begin{lem}\label{22}
		Recall  the integer $p_*$ from (\ref{p}), the finite root system $\dot{S}$ from Table~\ref{table11} and its base $\dot\D$ as above. Set
		\[\dot S^+(\dot\Delta):=(\Bigop{\b\in \dot\D}{}\bbbz^{\geq 0}\b)\cap (\dot S\setminus \{0\}).\]
		Suppose that $V$ is an integrable finite weight $\LL$-module. Then, there is a weight $\lam$ of $V$ with  \[\lam+\dot\a+rp_*\d\not \in \supp(V)\quad(r\in\bbbz, \dot\a\in \dot S^+(\dot\Delta)).\]
	\end{lem}

	\medskip
	Keep the same notation as above. We know that if $X\neq A(1,2{{n}}-1)^{(2)},$ we have $$\sspan_\bbbr R=\sspan_\bbbr R_0=\sspan_\bbbr\dot S\op \bbbr \d$$ while for $X {{=}} A(1,2{{n}}-1)^{(2)},$ we have \[\sspan_\bbbr R=\sspan_\bbbr R_0\op \bbbr \ep_1=\sspan_\bbbr\dot S\op \bbbr \ep_1\op \bbbr \d.\] Fix a functional  
	$\bs{g}:\sspan_\bbbr \dot\Pi(=\sspan_\bbbr \dot R)\longrightarrow \bbbr$
	such that 
	\begin{itemize}
		\item $\bs{g}(\dot\a)>0\; \hbox{for all } \dot\a\in \dot\Pi,$ in particular $\bs{g}(\d_p)>0$ for all $1\leq p\leq n$ and $\bs{g}(\e_i)>0$ for all $1\leq i\leq m$ if $X\not =A(1,2n-1)^{(2)},$
		\item   ${\rm min}\{\bs{g}(\ep_i)\mid 1\leq i\leq m\})>{\rm max}\{\bs{g}(\d_p)\mid 1\leq p\leq n\},$ for $m\neq 0$ and  $X\not =A(1,2n-1)^{(2)},$ 
		\item  $\bs{g}(\ep_1)>{\rm max}\{|\bs{g}(\d_p)|\mid 1\leq p\leq n\},$ for $X=A(1,2n-1)^{(2)}.$ 
	\end{itemize}
	We extend $\bs{g}$ to get a functional
	\[\bs{f}:\sspan_\bbbr R\longrightarrow \bbbr\] with 
	\begin{equation}\label{properties-functional}
		\begin{array}{l}
			\bullet ~~\bs{f}(\d)=0,\\
			\bullet ~~\hbox{for } \dot\a\in \dot R_{re}, \bs{f}(\dot\a)>0 \hbox{ if } \dot\a\in \sspan_{\bbbz^{\geq 0}}\dot \Pi, \hbox{ in particular if } \dot{\a}\in \dot S^+(\dot\Delta)\\
			\bullet~~\bs{f}(\ep_i\pm\d_p)>0  \quad(1\leq i\leq m,~1\leq p\leq n).
		\end{array}
	\end{equation}
	
	\begin{Pro}\label{main-mm}
		Recall 	$\dot{S}=\dot{R}\cap R_0$ and $\bs{f}$ as above as well as  $p_*$ as in (\ref{p}). Set 
		\[T:=R\cap \{\dot\a+r\d\mid \dot \a\in \dot R,0\leq r\leq p_*-1\}.\]
		Assume that $V$ is an integrable finite weight $\LL$-module of level zero. Then, there exists a nonzero weight vector $v$ such that $\LL^{\b+p_*r\d}v=\{0\}$ for all $r\in\bbbz$ and $\b\in T$ with $\bs{f}(\b)>0$. 
	\end{Pro}
	\pf
	Set
	\begin{align*}
		W:=\{v\in V\mid \LL^{\dot{\a}+sp_*\d}v=\{0\}\quad (\dot{\a}\in \dot S, s\in \bbbz, \bs{f}(\dot{\a})>0)\}.
	\end{align*}
	By Lemma \ref{22}, we know that $W$ is a nonzero $\cH$-module. For each nonzero weight vector $v\in W,$ define the finite set 
	\[\aa_v:=\{\b\in T\mid  \bs{f}(\b)>0~ \&~ \exists s\in \bbbz\hbox{ {s.t.} }~\LL^{\b+sp_*\d}v\neq \{0\}\} .\] 
	Suppose that $v_0$ is a nonzero weight vector in $W$ such that  $|\cA_{v_0}|$,  the cardinality $\cA_{v_0},$ 
	is minimum. We claim that $\cA_{v_0}=\emptyset$. We show this through the following two steps:
	
	\noindent{\bf Step~1.}
	If $\aa_{v_0}\neq\emptyset,$  ${{\b_*}}\in\cA_{v_0}$ with 	$\bs{f}({{\b_*}})=
	{\text{\rm{max}}}\{\bs{f}(\b)\mid \b\in \aa_{v_0}\},$ $0\neq v\in \LL^{{{\b_*}}+sp_*\d}v_0$ for some $s\in\bbbz,$ then $v\in W$ and $\aa_v=\aa_{v_0}:$ 
	To show that $v\in W,$ suppose that $r\in\bbbz$ and $\dot{\a}\in \dot S$ with $\bs{f}(\dot{\a})>0$,  then, we have
	\begin{equation}\label{mm-11}
		\LL^{\dot\a+rp_*\d}v\subseteq \LL^{\dot\a+rp_*\d}\LL^{{{\b_*}}+sp_*\d}v_0\subseteq\LL^{{{\b_*}}+sp_*\d}\LL^{\dot\a+rp_*\d}v_0+\LL^{{{\b_*}}+\dot{\a}+(s+r)p_*\d}v_0\subseteq \LL^{{{\b_*}}+\dot{\a}+(s+r)p_*\d}v_0.
	\end{equation}
	If $\LL^{\b_*+\dot{\a}+(s+r)p_*\d}v_0\neq\{0\}$, then		
	$\b_*+\dot\a\in T$. So $\b_*+\dot\a\in\cA_{v_0}$.
	But we have \[\bs{f}(\b_*+\dot\a)=\bs{f}(\b_*)+\bs{f}(\dot{\a})>\bs{f}(\b_*)\]  contradicting  the choice of $\b_*$.
	Therefore,  
	$\LL^{{{\b_*}}+\dot{\a}+(s+r)p_*\d}v_0=\{0\}$ which in turn implies that $\LL^{\dot\a+rp_*\d}v=\{0\}$ (see (\ref{mm-11})) and so $v\in W$ as we desired. Next we need to show that  $\aa_v=\aa_{v_0}.$ For this, it is enough to prove $\cA_v\sub\cA_{v_0}.$
	Assume $\gamma\in\cA_{v}$. So, there exists $r\in \bbbz$ such that
	$\LL^{\gamma+rp_*\d}v\neq\{0\}$. We note that \[[\LL^{\gamma+rp_*\d},\LL^{{{\b_*}}+sp_*\d}]\sub \LL^{\gamma+{{\b_*}}+(r+s)p_*\d}.\] If  $\gamma+{{\b_*}}+(r+s)p_*\d\not \in R$, then $ \LL^{\gamma+{{\b_*}}+(r+s)p_*\d}=\{0\}$. Also if 
	${{\b_*}}+\gamma+(r+s)p_*\d\in R$, we have ${{\b_*}}+\gamma+(r+s)p_*\d=\dot \eta+q\d+kp_*\d$ for some  $\dot\eta\in\dot R,$ $k\in\bbbz$ and $0\leq q\leq p_*-1$ with $\dot\eta+q\d\in T.$ Since 
	\[\bs{f}(\dot\eta+q\d)=\bs{f}({{\b_*}}+\gamma)>\bs{f}({{\b_*}}),\] we have $ \LL^{\gamma+{{\b_*}}+(r+s)p_*\d}v=\{0\}.$ Therefore, we have
	
	\begin{align*}
		\{0\}\neq \LL^{\gamma+rp_*\d}v\sub \LL^{\gamma+rp_*\d}\LL^{{{\b_*}}+sp_*\d}v_0\sub \LL^{{{\b_*}}+sp_*\d}\LL^{\gamma+rp_*\d}v_0+[\LL^{\gamma+rp_*\d},\LL^{{{\b_*}}+sp_*\d}]v_0=\LL^{{{\b_*}}+sp_*\d}\LL^{\gamma+rp_*\d}v_0.
	\end{align*}
	In particular,   $\LL^{\gamma+rp_*\d}v_0\neq\{0\}$, and so $\gamma\in \aa_{v_0}.$ This completes the proof of the first step.
	
	\noindent{\bf Step~2.} $\cA_{v_0}=\emptyset:$ To the contrary, assume $\cA_{v_0}\neq \emptyset$ and pick $\b_*\in\cA_{v_0}$ with $\bs{f}({{\b_*}})=
	{\text{\rm{max}}}\{\bs{f}(\b)\mid \b\in \aa_{v_0}\}.$
	Since $\b_*\in\cA_{v_0}$, there is an integer $s_0$ such that $\LL^{\b_*+s_0p_*\d}v_0\neq \{0\}$. We fix	$0\neq v_1\in \LL^{\b_*+s_0p_*\d}v_0$. Using Step~1, we have $v_1\in W$ and $\cA_{v_1}=\cA_{v_0}$. So, we have $\b_*\in\cA_{v_1}$. Therefore, there is an  integer $s_1$ such that 
	$\LL^{\b_*+s_1p_*\d}v_1\neq \{0\}$. Continuing this  process, we find integers 	$s_i$'s ($i=0,1,2,3,\cdots$) and nonzero weight vectors $v_i$'s such that 
	\[\b_{*}+s_ip_*\d\in R\quad (i=0,1,2,\ldots)\andd 0\neq v_i\in \LL^{\b_{*}+s_{i-1}p_*\d}v_{i-1}\quad (i\geq 1).\]
	Suppose that $\mu_0$ is the weight of $v_0$. Since  for	$r\geq 1$, $v_r$ is of weight  
	$$\mu_0+r\b_*+(s_0+\cdots+s_{r-1})p_*\d,$$
	we have
	\begin{equation*}\label{mur}
		\mu_r:=\mu_0+r\b_*+(s_0+\cdots+s_{r-1})p_*\d\in \supp(V)\quad\quad(r\in\bbbz^{>0}).
	\end{equation*} 
	By \cite[Lem.~3.1]{DKY},  $\b_*$ is real.	
	Pick $r_0\in\bbbz^{>0}$ with
	\begin{equation}\label{positive}
		\begin{split}
			2(\mu_{r_0},\b_*)/(\b_*,\b_*)=2(\mu_0+r_0\b_*,\b_*)/(\b_*,\b_*)>0.
		\end{split}
	\end{equation}
	Since $\b_*\in T\cap R_{re}$, for each positive integer $n$, we have
	$$\pm(\b_*+(-(r-2)s_{r_0}+s_{r_0+1}+\cdots+s_{r_0+r-1})p_*\d)\in R^{ln}(V).$$
	Also, as the level of the module under consideration is zero, we have $(\mu_r,\d)=0.$ So,  by \eqref{positive} and \eqref{in-supp 1}, we have
	
	\begin{align*}
		& \mu_{r_0}+(r-1)(\b_*+p_*s_{r_0}\d)\\
		&= \mu_{r_0}+r\b_*+(p_*s_{r_0}+\cdots+p_*s_{r_0+r-1})\d-(\b_*+(-(r-2)p_*s_{r_0}+p_*s_{r_0+1}+\cdots+p_*s_{r_0+r-1})\d)\\
		&= \mu_{r}-(\b_*+(-(r-2)s_{r_0}+s_{r_0+1}+\cdots+s_{r_0+r-1})p_*\d)\\
		&\in \text{{supp}}(V)\quad\quad (r\in \bbbz^{>0}).
	\end{align*}
	This contradicts the fact that $\b_*+p_*s_{r_0}\d\in R^{ln}(V)$ and so, we are done.\qed
	\bigskip
	
	The following theorem together with \S~\ref{zero-level-M} give a complete characterization of zero-level simple integrable finite weight $\LL$-modules if $\LL\neq A(2m,2n)^{(4)}$ and a complete characterization of zero-level simple admissible $\LL$-modules if $\LL= A(2m,2n)^{(4)}.$
	
	\begin{Thm}\label{MMain} 
		Suppose that $V$ is a simple integrable finite weight $\LL$-module of level zero. Then, there is a linear functional $\bs{f}:\sspan_\bbbr R\longrightarrow \bbbr$ such that $R^\circ_{\bs{f}}=\bbbz\d$ with $\Omega:=V^{\LL^+_{\bs f}}\neq\{0\}$. In particular, $\Omega$ is a simple finite weight module over $\Bigop{k\in\bbbz}{}\LL^{k\d}$ and $V={\rm ind}^\LL_{\bs{f}}(\Omega).$ \end{Thm}
	\pf
	Let $\dot{S}$ be the finite root system  defined in Table \ref{table11} and consider $\bs{f}$ as in  (\ref{properties-functional}). In particular, we have  $R_{\bs{f}}^\circ=\bbbz\d$. We know from Proposition~\ref{main-mm} that $\Omega=\{v\in V\mid \LL_{\bs{f}}^+ v=\{0\}\}$ is nonzero and so by Proposition~\ref{ind}, we are done.\qed

\end{document}